\documentclass[10pt]{amsart}
\usepackage{amsmath, amssymb}
\usepackage{mathrsfs}
\usepackage[poly,all]{xy}
\usepackage{texdraw}

\newtheorem{Thm}{Theorem}[section]
\newtheorem{Prop}[Thm]{Proposition}

\newtheorem{Lem}[Thm]{Lemma}

\theoremstyle{definition}
\newtheorem{Def}[Thm]{Definition}
\newtheorem{Exa}[Thm]{Example}

\numberwithin{equation}{section}

\newcommand{\Z}{\mathbb{Z}}

\newcommand{\C}{\mathbb{C}}

\newcommand{\g}{\mathfrak{g}}

\newcommand{\slm}{\mathfrak{sl}}
\newcommand{\HT}{{\rm ht}}
\newcommand{\wt}{{\rm wt}}

\newcommand{\ind}{{\rm Ind}}
\newcommand{\coind}{{\rm coInd}}
\newcommand{\res}{{\rm Res}}
\newcommand{\soc}{{\rm soc}}
\newcommand{\hd}{{\rm hd}}
\newcommand{\Hom}{{\rm Hom}}
\newcommand{\HOM}{{\rm HOM}}

\newcommand{\fMod}{ \mathrm{fmod}}
\newcommand{\e}{\tilde{e}}
\newcommand{\f}{\tilde{f}}
\newcommand{\ch}{{\rm ch}}
\newcommand{\qch}{{\rm ch}_{q}}

\newcommand{\YD}{ \mathcal{Y} }

\newcommand{\qdim}{ {\rm qdim} }
\newcommand{\infl}{ {\rm infl} }
\newcommand{\pr}{ {\rm pr} }

\newcommand{\ms}{ \nabla }
\newcommand{\tms}{ { {^t}\ms} }
\newcommand{\ims}{ \widehat{\nabla} }   % \blacktriangledown
\newcommand{\tims}{ { {^t}\ims} }

\begin{document}

\title[Irreducible Modules over KLR Algebras of type $A_n$ and Semistandard Tableaux]
{Irreducible Modules over Khovanov-Lauda-Rouquier \\ Algebras of
type $A_n$ and Semistandard Tableaux }
\author[Seok-Jin Kang]{Seok-Jin Kang$^{1,2}$}
%\thanks{$^*$ Corresponding author.}
\thanks{$^1$ This research was supported by KRF Grant \# 2007-341-C00001.}
\thanks{$^2$ This research was supported by National Institute for Mathematical Sciences (2010 Thematic Program, TP1004).}
\address{Department of Mathematical Sciences,
Seoul National University, Gwanak-ro 599, %San 56-1 Sillim-dong,
Gwanak-gu, Seoul 151-747, Korea} \email{sjkang@math.snu.ac.kr}
\author[Euiyong Park]{Euiyong Park$^{3,4}$}
\thanks{$^3$ This research was supported by BK21 Mathematical Sciences Division.}
\thanks{$^4$ This research was supported by NRF Grant \# 2010-0010753.}
\address{BK21 Mathematical Sciences Division, Seoul National University,
Gwanak-ro 599, %San 56-1 Sillim-dong,
Gwanak-gu, Seoul 151-747, Korea}
\email{pwy@snu.ac.kr}

%\date{Jan 12, 2009}
%\urladdr{}
%\dedicatory{}
%\thanks{This paper is in final form and no version of it will be submitted
%for publication elsewhere.}

\begin{abstract}

Using combinatorics of Young tableaux, we give an explicit
construction of irreducible graded modules over
Khovanov-Lauda-Rouquier algebras $R$ and their cyclotomic quotients
$R^{\lambda}$ of type $A_{n}$. Our construction is compatible with
crystal structure. Let ${\mathbf B}(\infty)$ and ${\mathbf
B}(\lambda)$ be the $U_q(\slm_{n+1})$-crystal consisting of
marginally large tableaux and semistandard tableaux of shape
$\lambda$, respectively. On the other hand, let ${\mathfrak
B}(\infty)$ and ${\mathfrak B}(\lambda)$ be the
$U_q(\slm_{n+1})$-crystals consisting of isomorphism classes of
irreducible graded $R$-modules and $R^{\lambda}$-modules,
respectively. We show that there exist explicit crystal isomorphisms
$\Phi_{\infty}: {\mathbf B}(\infty) \overset{\sim} \longrightarrow
{\mathfrak B}(\infty)$ and $\Phi_{\lambda}: {\mathbf B}(\lambda)
\overset{\sim} \longrightarrow {\mathfrak B}(\lambda)$.

\end{abstract}

\maketitle

%\tableofcontents

\section*{Introduction}

Let $\g$ be a symmetrizable Kac-Moody algebra and let $U_q^{-}(\g)$
be the negative part of the quantum group $U_q(\g)$ associated with
$\g$. Recently, Khovanov and Lauda \cite{KL08a, KL09} and Rouquier
\cite{R08} independently introduced a new family of graded algebras
$R$ whose representation theory gives a categorification of
$U_q^{-}(\g)$. The algebra $R$ is called the {\it
Khovanov-Lauda-Rouquier algebra} associated with $\g$. Let $\lambda
\in P^{+}$ be a dominant integral weight. It was conjectured that
the {\it cyclotomic quotient} $R^{\lambda}$ gives a categorification
of irreducible highest weight $U_q(\g)$-module $V(\lambda)$ with
highest weight $\lambda$ \cite{KL09}. This conjecture was shown to
be true when $\g$ is of type $A_{\infty}$ or $A_{n}^{(1)}$
\cite{BK08, BK09, BS08}.

In \cite{LV09}, Lauda and Vazirani investigated the crystal
structure on the set of isomorphism classes of finite dimensional
irreducible graded modules over $R$ and $R^{\lambda}$, where the
Kashiwara operators are defined in terms of induction and
restriction functors. Let ${\mathfrak B}(\infty)$ and ${\mathfrak
B}(\lambda)$ denote the $U_q(\g)$-crystal consisting of irreducible
graded $R$-modules and $R^{\lambda}$-modules, respectively. They
showed that there exist $U_q(\g)$-crystal isomorphisms ${\mathfrak
B}(\infty) \overset{\sim} \longrightarrow B(\infty)$ and ${\mathfrak
B}(\lambda) \overset{\sim} \longrightarrow B(\lambda)$, where
$B(\infty)$ and $B(\lambda)$ are the crystals of $U_q^{-}(\g)$ and
$V(\lambda)$, respectively. Consequently, every irreducible graded
module can be constructed {\it inductively} by applying the
Kashiwara operators on the trivial module.

On the other hand, in \cite{KR09}, Kleshchev and Ram gave an
explicit construction of irreducible graded $R$-modules for all
finite type using combinatorics of Lyndon words. They characterized
the irreducible graded $R$-modules as the simple heads of certain
induced modules. In \cite{HMM09}, Hill, Melvin and Mondragon
constructed cuspidal representations for all finite type and
completed the classification of irreducible graded $R$-modules given
in \cite{KR09}. It is still an open problem to construct irreducible
graded $R^{\lambda}$-modules in terms of Lyndon words. However, in
this approach, the action of Kashiwara operators is hidden in the
combinatorics of Lyndon words.

In this paper, using combinatorics of Young tableaux, we give an
explicit construction of irreducible graded $R$-modules and
$R^{\lambda}$-modules when $\g$ is of type $A_{n}$. Our construction
is compatible with crystal structure in the following sense. Let
${\mathbf B}(\lambda)$ be the set of all {\it semistandard tableaux}
of shape $\lambda$ with entries in $\{1, 2, \ldots, n+1\}$ and let
${\mathbf B}(\infty)$ be the set of all {\it marginally large
tableaux}. It is well-known that ${\mathbf B}(\lambda)$ and
${\mathbf B}(\infty)$ have $U_q(\slm_{n+1})$-crystal structures and
they are isomorphic to $B(\lambda)$ and $B(\infty)$, respectively
\cite{HK02, HL08, KN94, Lee07}. For each semistandard tableau of
shape $\lambda$ (resp.\ a marginally large tableau), we construct an
irreducible graded $R^{\lambda}$-module (resp.\ $R$-module) and show
that there exist explicit crystal isomorphisms  $\Phi_{\lambda}:
{\mathbf B}(\lambda) \overset{\sim} \longrightarrow {\mathfrak
B}(\lambda)$ and $\Phi_{\infty}: {\mathbf B}(\infty) \overset{\sim}
\longrightarrow {\mathfrak B}(\infty)$. In our construction,
irreducible graded modules appear as the simple heads of certain
induced modules that are determined by semistandard tableaux or
marginally large tableaux. Our work was inspired by \cite{KR09} and
\cite{V02}. We expect our work can be extended to other classical
type using combinatorics of Kashiwara-Nakashima tableaux in
\cite{KN94}. As was shown in \cite{HM09}, one may construct irreducible
modules over the Khovanov-Lauda-Rouquier algebra of type $A$ using cellular basis technique
introduced in \cite{GL96}.

This paper is organized as follows. In Section \ref{Sec: sl(n+1)}
and Section \ref{Sec: B(infty)}, we review the theory of
$U_q(\slm_{n+1})$-crystals and their combinatorial realization in
terms of Young tableaux. In Section \ref{Sec: KLR algebras}, we
recall the fundamental properties of Khovanov-Lauda-Rouquier
algebras $R$ and their cyclotomic quotients $R^\lambda$. We also
describe the crystal structures on $\mathfrak{B}(\infty)$ and
$\mathfrak{B}(\lambda)$. Section \ref{Sec: KL alg and B(lambda)} is
devoted to the main result of our paper. For each semistandard
tableau $T \in {\mathbf B}(\lambda)$, we construct an irreducible
graded $R^{\lambda}$-module $\Phi_{\lambda}(T):= \text{hd}
\text{Ind} \nabla_{T}$ as the simple heads of the induced module
$\text{Ind} \nabla_{T}$ determined by $T$, and  show that the
correspondence $T \longmapsto \text{hd} \text{Ind} \nabla_{T}$
defines a crystal isomorphism $\Phi_{\lambda}: {\mathbf B}(\lambda)
\overset{\sim} \longrightarrow {\mathfrak B}(\lambda)$. In Section
\ref{Sec: KL alg and B(infty)}, we extend the construction given in
Section \ref{Sec: KL alg and B(lambda)} to marginally large tableaux
to obtain an explicit construction of irreducible graded
$R$-modules. We also show that there exists an explicit crystal
isomorphism $\Phi_{\infty} : {\mathbf B}(\infty) \overset{\sim}
\longrightarrow {\mathfrak B}(\infty)$ induced by $\Phi_{\lambda}$.

\vskip 3em

\section{The crystal $B(\lambda)$ and Semistandard tableaux } \label{Sec: sl(n+1)}

In this section, we review the theory of $U_q(\slm_{n+1})$-crystals
and their connection with combinatorics of Young tableaux (see, for
example, \cite{HK02, KN94}). Let $I=\{1,2,\ldots, n\}$ and let
$$A = (a_{ij})_{i,j\in I} = \left(
\begin{matrix} 2 & -1 & 0 & \cdots & 0 \\ -1 & 2 & -1 & \cdots & 0 \\
\vdots &  & \ddots & & \vdots \\ 0 & \cdots & -1 & 2 & -1 \\ 0 & 0 &
\cdots & -1 & 2 \end{matrix} \right)$$ be the Cartan matrix of type
$A_{n}$. Set $P^{\vee} = \Z h_1 \oplus \cdots \oplus \Z h_n$,
$\mathfrak{h} = \C \otimes_{\Z} P^{\vee}$, and define the linear
functionals $\alpha_i, \varpi_i \in \mathfrak{h}^*$ $(i\in I)$ by
$$\alpha_i(h_j) = a_{ji} \qquad \varpi_i(h_j) = \delta_{ij} \ \ (i,
j \in I).$$ The $\alpha_i$ (resp.\ $\varpi_i$) are called the {\it
simple roots} (resp.\ {\it fundamental weights}). Set
$\Pi=\{\alpha_1, \ldots, \alpha_n\}$, $Q=\Z \alpha_1 \oplus \cdots
\oplus \Z \alpha_n$ and $P = \Z \varpi_1 \oplus \cdots \oplus \Z
\varpi_n$. The quadruple $(A, P^{\vee}, \Pi, P)$ is called the {\it
Cartan datum of type $A_{n}$}. The free abelian groups $P^{\vee}$,
$P$ and $Q$ are called the {\it dual weight lattice}, {\it weight
lattice}, and {\it root lattice}, respectively. We denote by $P^{+}
= \{\lambda \in P \mid \lambda(h_i) \ge 0 \ \text{for all} \ i\in
I\}$ the set of all {\it dominant integral weights}. Define
$$\epsilon_1 = \varpi_1, \quad \epsilon_{k+1} = \varpi_{k+1} -
\varpi_{k} (k \ge 1).$$ Then $\alpha_i = \epsilon_i -
\epsilon_{i+1}$, $P = \Z \epsilon_1 \oplus \cdots \oplus \Z
\epsilon_n$, and every dominant integral weight $\lambda = a_1
\varpi_1 + \cdots + a_n \varpi_n$ can be written as $\lambda =
\lambda_1 \epsilon_1 + \cdots + \lambda_n \epsilon_n$, where
$\lambda_i = a_i + \cdots + a_n$ $(i=1, \ldots, n)$.

Let $q$ be an indeterminate and for $m\ge n \ge 0$, define
$$[n]_q=\frac{q^{n} - q^{-n}}{q-q^{-1}}, \quad [n]_q! = [n]_q [n-1]_q \cdots [2]_q
[1]_q, \quad \left[\begin{matrix}  m \\ n \end{matrix} \right]_q =
\frac{[m]_q!}{[n]_q! [m-n]_q!}.$$

\begin{Def}
The {\it quantum special linear algebra} $U_q(\slm_{n+1})$ is the
associative algebra over $\C(q)$ generated by the elements $e_i,
f_i\ (i=1,\ldots,n)$ and $q^h\ (h \in P^{\vee})$ with the following
defining relations:
\begin{equation}
\begin{aligned}
& q^h q^{h'} = q^{h+h'} \quad \text{ for } h,h' \in P^{\vee}, \\
& q^h e_i q^{-h} = q^{\alpha_i(h)}e_i, \quad q^h f_i q^{-h} = q^{-\alpha_i(h)}f_i, \\
& e_if_j - f_je_i = \delta_{ij} \frac{ q^{h_i} - q^{- h_i} }{q - q^{-1}}, \\
& \sum_{r=0}^{1-a_{ij}} (-1)^k e_i^{(1-a_{ij}-r)} e_j e_i^{(r)}
=\sum_{r=0}^{1-a_{ij}} (-1)^k f_i^{(1-a_{ij}-r)} f_j f_i^{(r)}=0 \ \
(i \neq j).
\end{aligned}
\end{equation}
\end{Def}

Here, we use the notation $e_i^{(k)} = e_i^k / [k]_q!$, $f_i^{(k)} =
f_i^k / [k]_q!$.

For each $\lambda \in P^+$, there exists a unique irreducible
highest weight $U_q(\slm_{n+1})$-module $V(\lambda)$ with highest
weight $\lambda$. It was shown in \cite{Kash90, Kash91} that every
irreducible highest wight module $V(\lambda)$ has a crystal basis
$(L(\lambda), B(\lambda))$. The crystal $B(\lambda)$ can be thought
of as a basis at $q=0$ and most of combinatorial features of
$V(\lambda)$ are reflected on the structure of $B(\lambda)$.
Moreover, the crystal bases have very nice behavior with respect to
tensor product. The basic properties of crystal bases can be found
in \cite{HK02, Kash90, Kash91}, etc.

By extracting the standard properties of crystal bases, Kashiwara
introduced the notion of abstract crystals in \cite{Kash93}. An {\it
abstract crystal} is a set $B$ together with the maps $\wt: B \to
P$, $\tilde{e}_i, \tilde{f}_i: B \to B \sqcup \{0\}$,
$\varepsilon_i, \varphi_i: B \to \Z \cup \{ -\infty \}$ $(i\in I)$
satisfying certain conditions. The details on abstract crystals,
including the notion of strict morphism, embedding, isomorphism,
etc., can be found in \cite{HK02, Kash93}. We only give some
examples including the tensor product of abstract crystals.

\begin{Exa} \
\begin{enumerate}
\item Let $(L(\lambda), B(\lambda))$ be the crystal basis of the
highest weight module $V(\lambda)$ with highest weight $\lambda \in P^{+}$.
Then $B(\lambda)$ is a $U_q(\slm_{n+1})$-crystal.

\item Let $(L(\infty), B(\infty))$ be the crystal basis of $U_q^-(\slm_{n+1})$.
Then $B(\infty)$ is a $U_q(\slm_{n+1})$-crystal.

\item For $\lambda \in P$, let $\mathbf{T}^\lambda=\{ t_\lambda \}$ and define the maps
\begin{align}
& \wt(t_\lambda)=\lambda,\quad \tilde e_i t_\lambda = \tilde f_i t_\lambda = 0 \text{  for  } i\in I, \nonumber \\
& \varepsilon_i(t_\lambda) = \varphi_i(t_\lambda) = -\infty \text{  for  } i \in I. \nonumber
\end{align}
Then $\mathbf{T}^\lambda$ is a $U_q(\slm_{n+1})$-crystal.

\item Let $\mathbf{C} = \{c \}$ and define the maps
$$ \wt(c)=0,\quad \tilde e_i c= \tilde f_i c = 0, \quad \varepsilon_i(c) =  \varphi_i(c) = 0\ \  (i \in I).$$
Then $\mathbf{C}$ is a $U_q(\slm_{n+1})$-crystal.

\item
Let $B_1$, $B_2$ be crystals and set $B_1 \otimes B_2 = B_1 \times
B_2$. Define the maps
\begin{align*}
& \wt(b_1\otimes b_2) = \wt(b_1)+ \wt(b_2), \\
& \varepsilon_i(b_1\otimes b_2) = \max\{ \varepsilon_i(b_1),\ \varepsilon_i(b_2) - \langle h_i, \wt(b_1) \rangle \}, \\
& \varphi_i(b_1\otimes b_2) = \max\{ \varphi_i(b_2),\ \varphi_i(b_1) + \langle h_i, \wt(b_2) \rangle \}, \\
& \tilde e_i(b_1\otimes b_2) = \left\{
                                     \begin{array}{ll}
                                       \tilde e_i b_1 \otimes b_2 & \hbox{ if } \varphi_i(b_1) \ge \varepsilon_i(b_2), \\
                                       b_1 \otimes \tilde e_i b_2 & \hbox{ if } \varphi_i(b_1) < \varepsilon_i(b_2),
                                     \end{array}
                                   \right. \\
& \tilde f_i(b_1\otimes b_2) = \left\{
                                     \begin{array}{ll}
                                       \tilde f_i b_1 \otimes b_2 & \hbox{ if } \varphi_i(b_1) > \varepsilon_i(b_2), \\
                                       b_1 \otimes \tilde f_i b_2 & \hbox{ if } \varphi_i(b_1) \le \varepsilon_i(b_2).
                                     \end{array}
                                   \right.
\end{align*}
Then $B_1 \otimes B_2$ is a $U_q(\slm_{n+1})$-crystal.

\end{enumerate}
\end{Exa}

We now recall the connection between the theory of
$U_q(\slm_{n+1})$-crystals and combinatorics of Young tableaux. A
{\it Young diagram} $\lambda$ is a collection of boxes arranged in
left-justified rows with a weakly decreasing number of boxes in each
row. We denote by $\YD$ the set of all Young diagrams. If a Young
diagram $\lambda$ contains $N$ boxes, we write $\lambda \vdash N$
and $|\lambda|=N$. The number of rows in $\lambda$ will be denoted
by $l(\lambda)$. We denote by $^t \lambda$ denotes the Young
diagram obtained by flipping $\lambda$ over its main diagonal. We
usually identify a Young diagram $\lambda$ with the partition
$\lambda = (\lambda_1 \ge \lambda_2 \ge \ldots)$, where $\lambda_i$
is the number of boxes in the $i$th row of $\lambda$. Recall that a
dominant integral weight $\lambda = a_1 \varpi_1 + \cdots + a_n
\varpi_n$ can be written as $\lambda = \lambda_1 \epsilon_1 + \cdots
+ \lambda_n \epsilon_n$, where $\lambda_i = a_i + \cdots + a_n$
$(i=1, \ldots, n)$. Since $\lambda_1 \ge \lambda_2 \ge \cdots \ge
\lambda_n \ge 0$, we identify a dominant integral weight $\lambda
=a_1 \varpi_1 + \cdots + a_n \varpi_n$ with a partition $\lambda =
(\lambda_1 \ge \lambda_2 \ge \cdots \ge \lambda_n \ge 0)$.

A {\it tableau} $T$ of shape $\lambda$ is a filling of a Young
diagram $\lambda$ with numbers, one for each box. We say that a
tableau $T$ is {\it semistandard} if
\begin{enumerate}
\item the entries in each row are weakly increasing from left to right,
\item the entries in each column are strictly increasing from top to bottom.
\end{enumerate}
We denote by ${\mathbf B}(\lambda)$ the set of all semistandard
tableaux of shape $\lambda$ with entries in $\{1, 2, \ldots, n+1
\}$.

Let $\lambda=(\lambda_1 \ge \lambda_2 \ge \cdots \ge \lambda_s > 0)$
be a Young diagram with $l(\lambda)=s$ and $|\lambda|=\lambda_1 +
\cdots + \lambda_s =N$. It is well-known that ${\mathbf B}(\lambda)$
has a $U_q(\slm_{n+1})$-crystal structure and is isomorphic to the
crystal $B(\lambda)$. Let us briefly recall how to define the
crystal structure on ${\mathbf B}(\lambda)$. Let ${\mathbf B} =
{\mathbf B}(\varpi_1)$ be the crystal of the vector representation
$V(\varpi_1)$ given below.

\vskip 1em \qquad\qquad\qquad\qquad\quad $\ {\mathbf B}\ : \ $
\xymatrix{ {
\begin{tabular}{|c|} \hline  1 \\ \hline \end{tabular}}
 \ar[r]^1 &
 { \begin{tabular}{|c|} \hline  2 \\ \hline \end{tabular}}
  \ar[r]^2 & \cdots \ar[r]^{n-1} &
  { \begin{tabular}{|c|} \hline  $n$ \\ \hline \end{tabular}}
   \ar[r]^{n} &
   { \begin{tabular}{|c|} \hline  $n+1$ \\ \hline \end{tabular}}
 }.
\vskip 1em
\noindent

By the {\it Middle-Eastern reading}, we mean the reading of entries
of a semistandard tableau by moving across the rows from right to
left and from top to bottom. Thus we get an embedding $\Upsilon_{M}:
{\mathbf B}(\lambda) \rightarrow {\mathbf B}^{\otimes N}$ and one
can define a $U_q(\slm_{n+1})$-crystal structure on ${\mathbf
B}(\lambda)$ by the inverse of $\Upsilon_{M}$. On the other hand,
the {\it Far-Eastern reading} proceeds down the columns from top to
bottom and from right to left and yields an embedding $\Upsilon_{F}:
{\mathbf B}(\lambda) \rightarrow {\mathbf B}^{\otimes N}$, which
also defines a $U_q(\slm_{n+1})$-crystal structure on ${\mathbf
B}(\lambda)$. It is known that the crystal structure on ${\mathbf
B}(\lambda)$ does not depend on $\Upsilon_{M}$ or $\Upsilon_{F}$ and
that it is isomorphic to $B(\lambda)$ (see, for example,
\cite{HK02}), where the highest weight vector is given by
$$ T_\lambda \  = \   \begin{tabular}{|c| c |c| c| c| c| c| c|}
           \hline
             1 & { $\cdots$ } & { $\cdots$ } & 1 &  1 & 1  \\
            \hline
            2 & $\cdots$&  {$\cdots$} & 2     \\
            \cline{1-4 }
           $\vdots$ & $\vdots$   & $\vdots$    \\
            \cline{1-3 }
            $s$ & {$\cdots$}  \vline\ \   $s$    \\
           \cline{1-2 }
         \end{tabular}.
$$

For a semistandard tableau $T \in {\mathbf B}(\lambda)$, write
$$ \Upsilon_M(T) = { \begin{tabular}{|c|} \hline  $a^T_{1,\lambda_1}$ \\ \hline \end{tabular}}\ \
 \otimes \cdots \otimes\  { \begin{tabular}{|c|} \hline  $a^T_{1,1}$ \\ \hline \end{tabular}}\ \
\otimes \ { \begin{tabular}{|c|} \hline  $a^T_{2,\lambda_2}$  \\
\hline \end{tabular}}\ \ \otimes \cdots \otimes \ {
\begin{tabular}{|c|} \hline  $a^T_{2,1}$ \\ \hline \end{tabular}}\ \
\otimes \ \cdots \otimes \ { \begin{tabular}{|c|} \hline  $a^T_{s,
1}$ \\ \hline \end{tabular}}\ \ ,$$
where $a^T_{ij}$ is the entry in
the $j$th box of the $i$th row of $T$.
Define a map $\Psi_{\lambda}
: \mathbf{B}(\lambda) \to \YD^s$ by
\begin{align} \label{Eq: injection Psi}
\Psi_{\lambda}(T) := (\mu^{(1)}, \ldots, \mu^{(s)}),
\end{align}
where $ \mu^{(k)} = (a^T_{k,\lambda_k} - k, a^T_{k,\lambda_k-1} - k,
\ldots, a^T_{k, 1} - k )$ for $k = 1, \ldots, s$. Note that
$\mu^{(k)}$ could be the empty Young diagram $(0,0,\ldots)$ and that
$\Psi_{\lambda}$ is injective. Pictorially, $\Psi_{\lambda}(T) =
(\mu^{(1)}, \ldots, \mu^{(s)})$ can be visualized as follows:

\vskip 1.5em
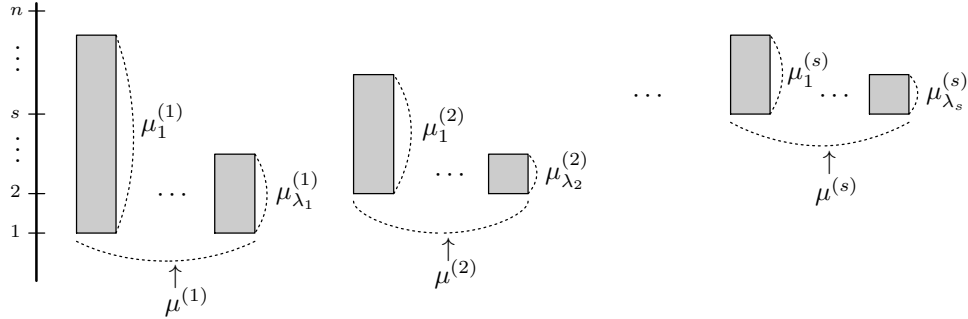
\begin{center}
\begin{texdraw}
\drawdim em \setunitscale 0.15  \linewd 0.6
\arrowheadtype t:F

\move(10 0) \lvec(10 -70)  % \move(0 -60) \lvec(200 -60)
\linewd 0.3
\move(8 -58) \lvec(12 -58)  \htext(3 -59){\scriptsize$1$}
\move(8 -48) \lvec(12 -48)  \htext(3 -49){\scriptsize$2$}

\move(8 -28) \lvec(12 -28)  \htext(3 -29){\scriptsize$s$}

  \htext(4 -17){$ \vdots$}
  \htext(4 -40){$ \vdots$}

\move(8 -2) \lvec(12 -2)  \htext(3 -3){\scriptsize$n$}

% \mu^(1)

\move(20 -58) \lvec(20 -8) \lvec(30 -8) \lvec(30 -58) \lvec(20 -58) \lfill f:0.8

\htext(40 -50){$\cdots$}

\move(55 -58) \lvec(55 -38) \lvec(65 -38) \lvec(65 -58) \lvec(55 -58) \lfill f:0.8

% \mu^(2)

\move(90 -48) \lvec(90 -18) \lvec(100 -18) \lvec(100 -48) \lvec(90 -48) \lfill f:0.8

\htext(110 -45){$\cdots$}

\move(124 -48) \lvec(124 -38) \lvec(134 -38) \lvec(134 -48) \lvec(124 -48) \lfill f:0.8

\htext(160 -25){$\cdots$}

% \mu^(4)
\move(185 -28) \lvec(185 -8) \lvec(195 -8) \lvec(195 -28) \lvec(185 -28) \lfill f:0.8

\htext(207 -25){$\cdots$}

\move(220 -28) \lvec(220 -18) \lvec(230 -18) \lvec(230 -28) \lvec(220 -28) \lfill f:0.8

%%%%%%%%%%%%%%%%%%%%%%%%%%%%%%%%%%%%%%%%%%%%%%%%%%%%%%%%%%%%%%%%%%%%%%%%%%

\lpatt(0.5 1)

% \mu^(1)_1
\move(30 -8)\clvec(30 -8)(40 -33)(30 -58)
\htext(36 -34){$\mu^{(1)}_1$}

% \mu^(1)_\lambda_1
\move(65 -38)\clvec(65 -38)(72 -48)(65 -58)
\htext(70 -52){$\mu^{(1)}_{\lambda_1}$}

% \mu^(2)_1
\move(100 -48)\clvec(100 -48)(110 -33)(100 -18)
\htext(107 -35){$\mu^{(2)}_1$}

% \mu^(2)_lamda_2
\move(134 -48)\clvec(134 -48)(139 -43)(134 -38)
\htext(138 -47){$\mu^{(2)}_{\lambda_2}$}

% \mu^(s)_1
\move(195 -28)\clvec(195 -28)(202 -18)(195 -8)
\htext(199 -21){$\mu^{(s)}_1$}

% \mu^(2)_lamda_2
\move(230 -28)\clvec(230 -28)(235 -23)(230 -18)
\htext(234 -27){$\mu^{(s)}_{\lambda_s}$}

% \mu^(1)
\move(20 -60)\clvec(30 -65)(40 -65)(43 -65)
\move(65 -60)\clvec(55 -65)(45 -65)(43 -65)
\htext(43 -72){$ \uparrow $}
\htext(42 -79){$\mu^{(1)}$}

% \mu^(2)
\move(90 -50)\clvec(90 -53)(100 -58)(112 -58)
\move(134 -50)\clvec(134 -53)(124 -58)(112 -58)
\htext(112 -65){$ \uparrow $}
\htext(110 -72){$\mu^{(2)}$}

% \mu^(4)
\move(185 -30)\clvec(185 -30)(195 -36)(208 -36)
\move(230 -30)\clvec(230 -30)(220 -36)(208 -36)
\htext(208 -43){$ \uparrow $}
\htext(206 -51){$\mu^{(s)}$}

\end{texdraw}
\end{center}
Here, $\mu^{(i)} = (\mu^{(i)}_1 \ge \mu^{(i)}_2 \ge \cdots \ge
\mu^{(i)}_{\lambda_i} \ge 0 )$ for $i=1, \ldots, s$. \vskip 1em

\begin{Exa} \label{Ex: lambda 1}
Let $\g = \slm_6$ and $\lambda = 2\varpi_1 + 2\varpi_2+ \varpi_4 +
\varpi_5$. If
$$ T = \begin{tabular}{|c|c|c|c|c|c|}
     \hline
     1 & 1 & 3 & 3 & 4 & 6  \\
     \hline
     2 & 3 & 4 & 5 \\
     \cline{1-4}
     3 & 5  \\
     \cline{1-2}
     5 & 6 \\
     \cline{1-2}
     6 \\
     \cline{1-1}
   \end{tabular},
 $$
 then
\begin{align*}
\Upsilon_M(T) &=
\begin{tabular}{|c|} \hline 6 \\ \hline \end{tabular} \otimes
\begin{tabular}{|c|} \hline 4 \\ \hline \end{tabular} \otimes
\begin{tabular}{|c|} \hline 3 \\ \hline \end{tabular} \otimes
\begin{tabular}{|c|} \hline 3 \\ \hline \end{tabular} \otimes
\begin{tabular}{|c|} \hline 1 \\ \hline \end{tabular} \otimes
\begin{tabular}{|c|} \hline 1 \\ \hline \end{tabular} \otimes
\begin{tabular}{|c|} \hline 5 \\ \hline \end{tabular} \otimes
\begin{tabular}{|c|} \hline 4 \\ \hline \end{tabular} \otimes
\begin{tabular}{|c|} \hline 3 \\ \hline \end{tabular} \otimes
\begin{tabular}{|c|} \hline 2 \\ \hline \end{tabular} \otimes
\begin{tabular}{|c|} \hline 5 \\ \hline \end{tabular} \otimes
\begin{tabular}{|c|} \hline 3 \\ \hline \end{tabular} \otimes
\begin{tabular}{|c|} \hline 6 \\ \hline \end{tabular} \otimes
\begin{tabular}{|c|} \hline 5 \\ \hline \end{tabular} \otimes
\begin{tabular}{|c|} \hline 6 \\ \hline \end{tabular}
\ , \\
\Upsilon_F(T) &=
\begin{tabular}{|c|} \hline 6 \\ \hline \end{tabular} \otimes
\begin{tabular}{|c|} \hline 4 \\ \hline \end{tabular} \otimes
\begin{tabular}{|c|} \hline 3 \\ \hline \end{tabular} \otimes
\begin{tabular}{|c|} \hline 5 \\ \hline \end{tabular} \otimes
\begin{tabular}{|c|} \hline 3 \\ \hline \end{tabular} \otimes
\begin{tabular}{|c|} \hline 4 \\ \hline \end{tabular} \otimes
\begin{tabular}{|c|} \hline 1 \\ \hline \end{tabular} \otimes
\begin{tabular}{|c|} \hline 3 \\ \hline \end{tabular} \otimes
\begin{tabular}{|c|} \hline 5 \\ \hline \end{tabular} \otimes
\begin{tabular}{|c|} \hline 6 \\ \hline \end{tabular} \otimes
\begin{tabular}{|c|} \hline 1 \\ \hline \end{tabular} \otimes
\begin{tabular}{|c|} \hline 2 \\ \hline \end{tabular} \otimes
\begin{tabular}{|c|} \hline 3 \\ \hline \end{tabular} \otimes
\begin{tabular}{|c|} \hline 5 \\ \hline \end{tabular} \otimes
\begin{tabular}{|c|} \hline 6 \\ \hline \end{tabular}
\ ,
\end{align*}
and $\Psi_{\lambda}(T) = (\mu^{(1)}, \mu^{(2)}, \mu^{(3)}, \mu^{(4)}, \mu^{(5)} )$, where
$$
 \mu^{(1)} := (5,3,2,2,0,0), \quad \mu^{(2)} := (3,2,1,0), \quad
\mu^{(3)} := (2,0), \quad \mu^{(4)} := (2,1), \quad \mu^{(5)} := (1).
$$
Pictorially, $\Psi_{\lambda}(T) = (\mu^{(1)}, \mu^{(2)}, \mu^{(3)},
\mu^{(4)}, \mu^{(5)} )$ is given as follows:

\vskip 1em
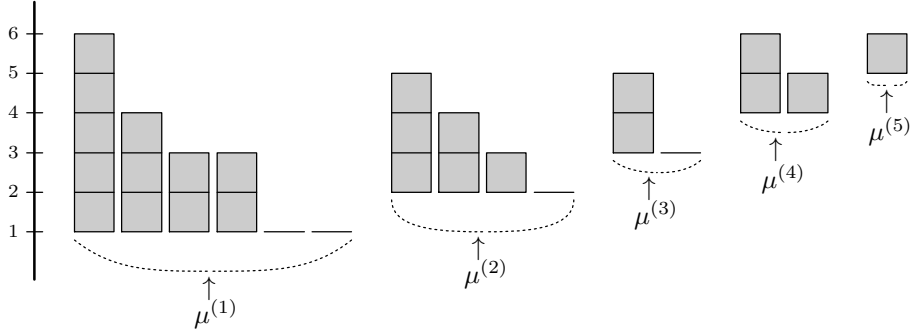
\begin{center}
\begin{texdraw}
\drawdim em \setunitscale 0.15  \linewd 0.6
\arrowheadtype t:F

\move(10 0) \lvec(10 -70)  % \move(0 -60) \lvec(200 -60)
\linewd 0.3
\move(8 -58) \lvec(12 -58)  \htext(3 -59){\scriptsize$1$}
\move(8 -48) \lvec(12 -48)  \htext(3 -49){\scriptsize$2$}
\move(8 -38) \lvec(12 -38)  \htext(3 -39){\scriptsize$3$}
\move(8 -28) \lvec(12 -28)  \htext(3 -29){\scriptsize$4$}
\move(8 -18) \lvec(12 -18)  \htext(3 -19){\scriptsize$5$}
\move(8 -8) \lvec(12 -8)  \htext(3 -9){\scriptsize$6$}

% \mu^(1)

\move(20 -58) \lvec(20 -8) \lvec(30 -8) \lvec(30 -58) \lvec(20 -58) \lfill f:0.8
\move(20 -48) \lvec(30 -48) \move(20 -38) \lvec(30 -38) \move(20 -28) \lvec(30 -28) \move(20 -18) \lvec(30 -18)
\move(32 -58) \lvec(32 -28) \lvec(42 -28) \lvec(42 -58) \lvec(32 -58) \lfill f:0.8
\move(32 -48) \lvec(42 -48) \move(32 -38) \lvec(42 -38)
\move(44 -58) \lvec(44 -38) \lvec(54 -38) \lvec(54 -58) \lvec(44 -58) \lfill f:0.8
\move(44 -48) \lvec(54 -48)
\move(56 -58) \lvec(56 -38) \lvec(66 -38) \lvec(66 -58) \lvec(56 -58) \lfill f:0.8
\move(56 -48) \lvec(66 -48)
\move(68 -58) \lvec(78 -58)
\move(80 -58) \lvec(90 -58)

% \mu^(2)

\move(100 -48) \lvec(100 -18) \lvec(110 -18) \lvec(110 -48) \lvec(100 -48) \lfill f:0.8
\move(100 -38) \lvec(110 -38) \move(100 -28) \lvec(110 -28)
\move(112 -48) \lvec(112 -28) \lvec(122 -28) \lvec(122 -48) \lvec(112 -48) \lfill f:0.8
\move(112 -38) \lvec(122 -38)
\move(124 -48) \lvec(124 -38) \lvec(134 -38) \lvec(134 -48) \lvec(124 -48) \lfill f:0.8
\move(136 -48) \lvec(146 -48)

% \mu^(3)
\move(156 -38) \lvec(156 -18) \lvec(166 -18) \lvec(166 -38) \lvec(156 -38) \lfill f:0.8
\move(156 -28) \lvec(166 -28)
\move(168 -38) \lvec(178 -38)

% \mu^(4)
\move(188 -28) \lvec(188 -8) \lvec(198 -8) \lvec(198 -28) \lvec(188 -28) \lfill f:0.8
\move(188 -18) \lvec(198 -18)
\move(200 -28) \lvec(200 -18) \lvec(210 -18) \lvec(210 -28) \lvec(200 -28) \lfill f:0.8

% \mu^(5)
\move(220 -18) \lvec(220 -8) \lvec(230 -8) \lvec(230 -18) \lvec(220 -18) \lfill f:0.8

% \mu^(1)
\lpatt(0.5 1)
\move(20 -60)\clvec(30 -68)(40 -68)(54 -68)
\move(90 -60)\clvec(80 -68)(70 -68)(55 -68)
\htext(52 -75){$ \uparrow $}
\htext(50 -82){$\mu^{(1)}$}

% \mu^(2)
\move(100 -50)\clvec(100 -58)(115 -58)(123 -58)
\move(146 -50)\clvec(146 -58)(131 -58)(124 -58)
\htext(120 -65){$ \uparrow $}
\htext(118 -72){$\mu^{(2)}$}

% \mu^(3)
\move(156 -40)\clvec(156 -40)(159 -43)(167 -43)
\move(178 -40)\clvec(178 -40)(175 -43)(168 -43)
\htext(163 -50){$ \uparrow $}
\htext(161 -57){$\mu^{(3)}$}

% \mu^(4)
\move(188 -30)\clvec(188 -30)(191 -33)(199 -33)
\move(210 -30)\clvec(210 -30)(207 -33)(200 -33)
\htext(195 -40){$ \uparrow $}
\htext(193 -48){$\mu^{(4)}$}

% \mu^(4)
\move(220 -20)\clvec(221 -21)(221 -21)(225 -21)
\move(230 -20)\clvec(229 -21)(229 -21)(226 -21)
\htext(223 -28){$ \uparrow $}
\htext(220 -36){$\mu^{(5)}$}

\end{texdraw}
\end{center}

\end{Exa}

\vskip 1em Note that $\Upsilon_M(T)$ can be obtained by reading the
top entries of columns in the above diagram from left to right.

The following lemma will play a crucial role in proving our main
result (Theorem \ref{Thm: crystal iso}).

\begin{Lem} \label{Lem: tableau}
Let $T$ be a semistandard tableau of shape $\lambda = (\lambda_1 \ge
\lambda_2 \ge  \ldots \ge\lambda_s > 0 )$, and let
$$\Psi_\lambda(T) = ( \mu^{(1)}, \mu^{(2)}, \ldots,  \mu^{(s)}  ), $$
where $ \mu^{(i)} = (\mu_1^{(i)} \ge \mu_2^{(i)} \ge \ldots \ge
\mu_{\lambda_i}^{(i)} \ge 0)$ for $i=1,\ldots,s$. Suppose that $T$
is not the highest weight vector $T_\lambda$; i.e., not all
$\mu^{(1)}, \ldots, \mu^{(s)} $ are $(0,0,\ldots)$. Set
\begin{equation*}
\begin{aligned}
 i_T & = \min
\{ \mu_j^{(i)} + i-1 |\  1 \le i \le s,\ 1\le j \le \lambda_i,\
\mu_j^{(i)} > 0 \}, \\
\varepsilon & = \varepsilon_{i_{T}}(T). \end{aligned}
\end{equation*}
Then we have
\begin{enumerate}
\item $ \varepsilon_{i_T}(T) = \# \{ \mu_j^{(i)} |\ \mu_j^{(i)} >0,\ \mu_j^{(i)} + i -1 = i_T,\ 1 \le i \le s,\ 1\le j \le \lambda_i
\}$;
\item $$ \e_{i_T}^{\varepsilon}(T) = T^+, $$
where $T^+$ is the tableau of shape $\lambda$ obtained from $T$ by replacing all entries $i_T+1$ by $i_T$ from the top row to the $i_T$th row.
\end{enumerate}

\end{Lem}
\begin{proof}
Let
$$ \Upsilon_M(T) = { \begin{tabular}{|c|} \hline  $a_{1,\lambda_1}$ \\ \hline \end{tabular}}\ \  \otimes \cdots \otimes\  { \begin{tabular}{|c|} \hline  $a_{1,1}$ \\ \hline \end{tabular}}\ \  \otimes \ { \begin{tabular}{|c|} \hline  $a_{2,\lambda_2}$  \\ \hline \end{tabular}}\ \ \otimes \cdots \otimes \ { \begin{tabular}{|c|} \hline  $a_{2,1}$ \\ \hline \end{tabular}}\ \ \otimes \ \cdots \otimes \ { \begin{tabular}{|c|} \hline  $a_{s, 1}$ \\ \hline \end{tabular}}\ \ ,$$
where $a_{ij}$ is the entry in the $j$th box of the $i$th row of
$T$. Then, from the definition of $\Psi_\lambda$, we have
 $$a_{i, \lambda_i - j + 1} = \mu_j^{(i)} + i \qquad (1 \le j \le \lambda_i),$$
 which yields
\begin{align*}
& i_T = \min \{ a_{ij}-1 |\  1 \le i \le s,\ 1\le j \le \lambda_i,\ a_{ij} > i \}, \\
& \# \{ a_{ij} |\ a_{ij} > i,\ a_{ij} -1 = i_T  \} = \# \{ \mu_j^{(i)} |\ \mu_j^{(i)} > 0, \mu_j^{(i)} + i -1 = i_T  \} .
\end{align*}
Note that the set $\{ a_{ij}-1 |\  1 \le i \le s,\ 1\le j \le
\lambda_i,\ a_{ij} > i \}$ is not empty since $T$ is not the highest
weight vector $T_\lambda$. Take the rightmost number $a_{pq}$ of
$\Upsilon_M(T)$ such that $a_{pq} = i_T+1$. Since $T$ is
semistandard; i.e.,
\begin{align*}
a_{pk} & \ge  a_{pq}   \quad \text{ for } k \ge q,\\
a_{p' q'}  &\ne   a_{pq} - 1  \quad \text{ for } 1 \le p' < p,
\end{align*}
and
$$ \varepsilon_i( { \begin{tabular}{|c|} \hline  $j$ \\ \hline \end{tabular}}\ ) = \left\{
                                                                                 \begin{array}{ll}
                                                                                   1 & \hbox{if } i=j-1, \\
                                                                                   0 & \hbox{otherwise,}
                                                                                 \end{array}
                                                                               \right.
\qquad \varphi_i( { \begin{tabular}{|c|} \hline  $j$ \\ \hline
\end{tabular}}\ ) = \left\{
                                                                                 \begin{array}{ll}
                                                                                   1 & \hbox{if } i=j, \\
                                                                                   0 &
                                                                                   \hbox{otherwise},
                                                                                 \end{array}
                                                                               \right.
$$
our assertion follows from the tensor product rule of crystals.
\end{proof}

\begin{Exa} \label{Ex: lambda 2}
We use the same notations as in Example \ref{Ex: lambda 1}. Consider
the following diagram for $\Psi_\lambda(T)$.

\vskip 1em
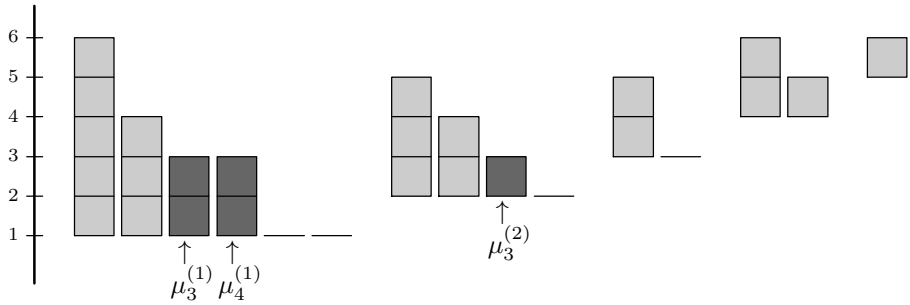
\begin{center}
\begin{texdraw}
\drawdim em \setunitscale 0.15  \linewd 0.6
\arrowheadtype t:F

\move(10 0) \lvec(10 -70)  % \move(0 -60) \lvec(200 -60)
\linewd 0.3
\move(8 -58) \lvec(12 -58)  \htext(3 -59){\scriptsize$1$}
\move(8 -48) \lvec(12 -48)  \htext(3 -49){\scriptsize$2$}
\move(8 -38) \lvec(12 -38)  \htext(3 -39){\scriptsize$3$}
\move(8 -28) \lvec(12 -28)  \htext(3 -29){\scriptsize$4$}
\move(8 -18) \lvec(12 -18)  \htext(3 -19){\scriptsize$5$}
\move(8 -8) \lvec(12 -8)  \htext(3 -9){\scriptsize$6$}

% \mu^(1)

\move(20 -58) \lvec(20 -8) \lvec(30 -8) \lvec(30 -58) \lvec(20 -58) \lfill f:0.8
\move(20 -48) \lvec(30 -48) \move(20 -38) \lvec(30 -38) \move(20 -28) \lvec(30 -28) \move(20 -18) \lvec(30 -18)
\move(32 -58) \lvec(32 -28) \lvec(42 -28) \lvec(42 -58) \lvec(32 -58) \lfill f:0.8
\move(32 -48) \lvec(42 -48) \move(32 -38) \lvec(42 -38)
\move(44 -58) \lvec(44 -38) \lvec(54 -38) \lvec(54 -58) \lvec(44 -58) \lfill f:0.4
\move(44 -48) \lvec(54 -48)
\move(56 -58) \lvec(56 -38) \lvec(66 -38) \lvec(66 -58) \lvec(56 -58) \lfill f:0.4
\move(56 -48) \lvec(66 -48)
\move(68 -58) \lvec(78 -58)
\move(80 -58) \lvec(90 -58)

% \mu^(2)

\move(100 -48) \lvec(100 -18) \lvec(110 -18) \lvec(110 -48) \lvec(100 -48) \lfill f:0.8
\move(100 -38) \lvec(110 -38) \move(100 -28) \lvec(110 -28)
\move(112 -48) \lvec(112 -28) \lvec(122 -28) \lvec(122 -48) \lvec(112 -48) \lfill f:0.8
\move(112 -38) \lvec(122 -38)
\move(124 -48) \lvec(124 -38) \lvec(134 -38) \lvec(134 -48) \lvec(124 -48) \lfill f:0.4
\move(136 -48) \lvec(146 -48)

% \mu^(3)
\move(156 -38) \lvec(156 -18) \lvec(166 -18) \lvec(166 -38) \lvec(156 -38) \lfill f:0.8
\move(156 -28) \lvec(166 -28)
\move(168 -38) \lvec(178 -38)

% \mu^(4)
\move(188 -28) \lvec(188 -8) \lvec(198 -8) \lvec(198 -28) \lvec(188 -28) \lfill f:0.8
\move(188 -18) \lvec(198 -18)
\move(200 -28) \lvec(200 -18) \lvec(210 -18) \lvec(210 -28) \lvec(200 -28) \lfill f:0.8

% \mu^(5)
\move(220 -18) \lvec(220 -8) \lvec(230 -8) \lvec(230 -18) \lvec(220 -18) \lfill f:0.8

% \mu_3^(1)
\htext(46 -65){$ \uparrow $}
\htext(44 -74){$\mu_3^{(1)}$}

% \mu_4^(1)
\htext(58 -65){$ \uparrow $}
\htext(56 -74){$\mu_4^{(1)}$}

% \mu_3^(2)
\htext(126 -55){$ \uparrow $}
\htext(124 -64){$\mu_3^{(2)}$}

\end{texdraw}
\end{center}

\vskip 1em

\noindent Thus we have
$$ i_T = 2,\qquad  \varepsilon_{i_T}(T) = 3 = \# \{ \mu^{(1)}_3, \mu^{(1)}_4, \mu^{(2)}_3  \}, $$
and
$$ \e_{2}^3 T = T^+ = \begin{tabular}{|c|c|c|c|c|c|}
     \hline
     1 & 1 & 2 & 2 & 4 & 6  \\
     \hline
     2 & 2 & 4 & 5 \\
     \cline{1-4}
     3 & 5  \\
     \cline{1-2}
     5 & 6 \\
     \cline{1-2}
     6 \\
     \cline{1-1}
   \end{tabular}
 .$$
\end{Exa}

\vskip 3em

\section{The crystal $B(\infty)$ and marginally large tableaux } \label{Sec: B(infty)}

In this section, we recall the realization of the
$U_q(\slm_{n+1})$-crystal $B(\infty)$ in terms of {\it marginally
large} tableaux given in \cite{Cli98, HL08, Lee07}.

\begin{Def}\
\begin{enumerate}
\item A semistandard tableau $T \in \mathbf{B}(\lambda)$ is {\it large} if it consists of $n$ non-empty rows, and if for each $i=1,\ldots, n$,
the number of boxes having the entry $i$ in the $i$th row is
strictly greater than the number of all boxes in the $(i+1)$th row.
\item A large tableau $T$ is {\it marginally large} if for each $i=1,\ldots,n$, the number of boxes having the entry $i$ in the $i$th row is
greater than the number of all boxes in the $(i+1)$th row by exactly one. In particular, the $n$th row of $T$ should contain one box having the entry $n$.
\end{enumerate}
\end{Def}

We consider the following tableau:
$$ T_0 :=
\begin{tabular}{|c|}
   \hline
   1  \\
   \hline
   2  \\
   \hline
   $ \vdots$  \\
   \hline
   $n$ \\
   \hline
 \end{tabular}\ .
  $$
  \vskip 1em
\noindent
For each marginally large tableau $T$, we construct a left-infinite extension of $T$ obtained by adding infinitely many copies of $T_0$
to the left of $T$. When there is no danger of confusion, we identity a marginally large tableau $T$ with the left-infinite extension of $T$.

\begin{Exa}
 Let $\g = \slm_{4}$. The following tableau $T$ is marginally large:
 $$ T =
 \begin{tabular}{|c|c|c|c|c|c|c|}
   \hline
   1 & 1 & 1 & 1 & 2 & 3 & 4 \\
   \hline
   2 & 2 & 2  \\
   \cline{1-3}
   3 & 4  \\
   \cline{1-2}
 \end{tabular}\ .
 $$
 The left-infinite extension of $T$ obtained by adding infinitely many copies of $T_0$
to the left of $T$ is given as follows.
$$ \begin{tabular}{c|c|c|c|c|c|c|c|c|}
   \hline
   $\cdots$ & 1 & 1 & 1 & 1 & 1 & 2 & 3 & 4 \\
   \hline
   $\cdots$ & 2 & 2 & 2 & 2  \\
   \cline{1-5}
   $\cdots$ & 3 & 3 & 4  \\
   \cline{1-4}
 \end{tabular}\ .
 $$
 \end{Exa}
 \vskip 1em

Let $\mathbf{B}(\infty)$ be the set of all left-infinite extensions
of marginally large tableaux. The Kashiwara operators $\f_i, \e_i \
(i\in I)$ on $\mathbf{B}(\infty)$ are defined as follows
(\cite{Lee07}):

\begin{enumerate}
\item[(B1)] We consider the infinite sequence of entries obtained by taking the Far-Eastern reading of $T \in \mathbf{B}(\infty)$.
To each entry $b$ in this sequence, we assign $-$ if $b = i+1$ and
$+$ if $b = i$. Otherwise we put nothing. From this sequence of
$+$'s and $-$'s, cancel out all $(+,-)$ pairs. The remaining
sequence is called the {\it $i$-signature} of $T$.
\item[(B2)] Denote by $T'$ the tableau obtained from $T$ by replacing the entry $i$ by $i+1$ corresponding to the leftmost $+$ in the $i$-signature of $T$.
\begin{itemize}
\item If $T'$ is marginally large, then we define $\f_i T$ to be $T'$.
\item If $T'$ is not marginally large, then define $\f_i T$ to be the marginally large tableau obtained by
pushing all the rows appearing below the changed box in $T'$ to the
left by one box.
\end{itemize}
\item[(B3)] Denote by $T''$ the tableau obtained from $T$ by replacing the entry $i$ by $i-1$ corresponding to the rightmost $-$ in the $i$-signature of $T$.
\begin{itemize}
\item If $T''$ is marginally large, then we define $\e_i T$ to be $T''$.
\item If $T''$ is not marginally large, then define $\e_i T$ to be the marginally large tableau obtained by
pushing all the rows appearing below the changed box in $T''$ to the
right by one box.
\end{itemize}
\item[(B4)] If there is no $-$ in the $i$-signature of $T$, we define $\e_i T = 0 $.
\end{enumerate}

Let $T$ be a marginally large tableau in $\mathbf{B}(\infty)$. For
each $i = 1, \ldots, n$, suppose that the $i$th row of $T$ contains
$b^i_j$-many $j$'s and infinitely many $i$'s. Define the maps $\wt:
\mathbf{B}(\infty) \to P$, $\varphi_i, \varepsilon_i:
\mathbf{B}(\infty) \to \Z$ by
\begin{align*}
\wt(T) &:= - \sum_{j=1}^n( \sum_{k=j+1}^{n+1}b_k^1 + \sum_{k=j+1}^{n+1}b_k^2 + \cdots + \sum_{k=j+1}^{n+1}b_k^j) \alpha_j, \\
\varepsilon_i(T) &:= \text{ the number of $-$'s in the $i$-signature of $T$}, \\
\varphi_i(T) &:= \varepsilon_i(T) + \langle h_i, \wt(T) \rangle.
\end{align*}

\begin{Prop} \cite[Theorem 4.8]{Lee07}
The sextuple $(\mathbf{B}(\infty), \wt, \e_i, \f_i, \varepsilon_i,
\varphi_i )$ becomes a $U_q(\slm_{n+1})$-crystal, which is
isomorphic to the crystal $B(\infty)$ of $U_q^-(\slm_{n+1})$.
\end{Prop}
Note that the highest weight vector $T_\infty$ of
$\mathbf{B}(\infty)$ is given as follows:
$$
T_\infty =\ \begin{tabular}{c|c|c|c|c|}
             \hline
             $\cdots$ & 1 & 1 & 1 & 1 \\
             \hline
            $\cdots$  & 2 & 2 & 2  \\
             \cline{1-4}
            $\vdots$  & $ \vdots$ & $ \vdots$ \\
              \cline{1-3}
            $\cdots$  & $n$  \\
             \cline{1-2}
           \end{tabular}
$$

It was shown in \cite{Kash93} that there is a unique strict crystal embedding
\begin{align} \label{Eq: Crystal embedding}
\iota_{\lambda} : \mathbf{B}(\lambda) \hookrightarrow
\mathbf{B}(\infty) \otimes \mathbf{T}^{\lambda} \otimes \mathbf{C}
\quad \text{given by} \quad T_{\lambda} \mapsto T_\infty \otimes
t_{\lambda} \otimes c,
\end{align}
where $T_\lambda$ is the highest weight vector of
$\mathbf{B}(\lambda)$. We now describe this crystal embedding
explicitly. Let $T $ be a semistandard tableau of $
\mathbf{B}(\lambda)$. We consider the left-infinite extension $T'$
obtained from $T$ by adding infinitely many copies of $T_0$
to the left of $T$. Then we construct the marginally large tableau
$T_{\rm ml}$ from $T'$ by shifting the rows of $T'$ in an
appropriate way. Note that $T_{\rm ml}$ is uniquely determined. For
example, if
 $$
 T= \
\begin{tabular}{|c|c|c|c|c|}
  \hline
  1 & 1 & 2 & 2 & 3 \\
  \hline
  2 & 3 & 3  \\
  \cline{1-3}
  4\\
  \cline{1-1}
\end{tabular}\ ,
 $$
 then we have
 $$
 T_{\rm ml} = \
\begin{tabular}{c|c|c|c|c|c|c|c|c|c|}
  \hline
  $\cdots$ & 1 & 1 & 1 & 1 & 1 & 1 & 2 & 2 & 3 \\
  \hline
  $\cdots$ & 2 & 2 & 2 & 3 & 3  \\
  \cline{1-6}
  $\cdots$ & 3& 4\\
  \cline{1-3}
\end{tabular}.
 $$

 Now the crystal embedding $\iota_\lambda: {\mathbf B}(\lambda)
 \hookrightarrow {\mathbf B}(\infty) \otimes {\mathbf
 T}^{\lambda} \otimes {\mathbf C}$ is given by $T \longmapsto T_{\rm
 ml} \otimes t_{\lambda} \otimes c$ \cite{Lee07}.
\vskip 1em

For $T \in \mathbf{B}(\infty)$, we denote by $a_{ij}^T$ the entry in
the $j$th box from the right in the $i$th row of $T$. Define a map
$\Psi_{\infty}:\mathbf{B}(\infty) \to \mathcal{Y}^n$ by
\begin{align} \label{Eq: injection Psi infity}
\Psi_{\infty}(T) := (\mu^{(1)}, \ldots, \mu^{(n)} ),
\end{align}
where $\mu^{(k)} = ( a_{k,1}^T -k, a_{k,2}^T - k , \ldots )$ for $k=1, \ldots, n$. Since
$$a_{k,j}^T - k = 0 \quad \text{ for } j\gg 0 ,$$
the Young diagram $\mu^{(k)}$ is well-defined for each $k$.
Then, by construction, for any $T\in \mathbf{B}(\lambda)$, we have
$$\Psi_{\lambda}(T) = \Psi_{\infty}(T_{\rm ml})$$
up to adding the empty Young diagrams. More precisely, we have the following lemma.

\begin{Lem} \label{Lem: Psi B(infty)}
Let $T $ be a semistandard tableau of $\mathbf{B}(\lambda)$, and $\iota_\lambda(T) = T_{\rm ml} \otimes t_\lambda \otimes c$.
Then $\Psi_{\infty}( T_{\rm ml} )$ is the $n$-tuple of Young diagrams obtained from $\Psi_{\lambda}(T)$ by adding the empty Young diagrams.
\end{Lem}

\vskip 3em

\section{Khovanov-Lauda-Rouquier algebras of type $A_n$} \label{Sec: KLR algebras}

In this section, we review the basic properties of
Khovanov-Lauda-Rouquier algebras \cite{KL08a, KL09, LV09, R08}. Let
$\alpha, \beta \in Q^+$ and $d = \HT(\alpha), d' = \HT(\beta)$.
Define $$ I^{\alpha} := \{  \mathbf{i} = (i_1, \ldots, i_d) \in
I^d|\ \alpha_{i_1} + \cdots + \alpha_{i_d} = \alpha  \}. $$ Then the
symmetric group $\Sigma_d$ acts on $I^{\alpha}$ naturally. Let
$\Sigma_{d + d' }/ \Sigma_{d} \times \Sigma_{d'}$ be the set of the
minimal length coset representatives of $\Sigma_{d} \times
\Sigma_{d'}$ in $\Sigma_{d + d' }$. The following proposition is
well-known.

\begin{Prop}\label{Prop: shuffle} \cite[Chapter 2.1]{GP00},
\cite[Section 2.2]{LV09}

There is a 1-1- correspondence between $\Sigma_{d + d' }/ \Sigma_{d}
\times \Sigma_{d'}$ and the set of all shuffles of $\mathbf{i}$ and
$\mathbf{j}$, where  $\mathbf{i} = (i_1, \ldots, i_d)  \in I^\alpha$
and $ \mathbf{j} = (j_1, \ldots, j_{d'}) \in I^\beta$.
\end{Prop}

For $\mathbf{i} = (i_1, \ldots, i_d)  \in I^\alpha$ and $ \mathbf{j}
= (j_1, \ldots, j_{d'}) \in I^\beta$, we denote by $\mathbf{i}
* \mathbf{j}$ the concatenation of $\mathbf{i}$ and $\mathbf{j}$:
$$ \mathbf{i} * \mathbf{j} := (i_1, \ldots, i_d, j_1, \ldots, j_{d'}) \in I^{\alpha + \beta}.  $$

\begin{Def} Let $\alpha \in Q^+$ and $d = \HT(\alpha)$.
The {\it Khovanov-Lauda-Rouquier algebra} $R(\alpha)$ of type $A_n$
corresponding to $\alpha \in Q^+$ is the associative graded
$\C$-algebra generated by $1_{\mathbf{i}}\ (\mathbf{i} \in
I^{\alpha})$, $x_k\ (1\le k\le d), \tau_t\ (1\le t\le d-1) $ with
the following defining relations :
\begin{equation}\label{Eq:KLRA}
\begin{aligned}
& 1_{\mathbf{i}} 1_{\mathbf{j}} = \delta_{\mathbf{ij}} 1_{\mathbf{i}},\ x_k 1_\mathbf{i} =  1_\mathbf{i} x_k, \ \tau_t 1_{\mathbf{i}} = 1_{\tau_t( \mathbf{i})} \tau_t,\\
& x_k x_l = x_l x_k, \\
& \tau_t \tau_s = \tau_s \tau_t \text{ if } |t - s| > 1, \\
& \tau_t \tau_t 1_{\mathbf{i}} = \left\{
                                       \begin{array}{ll}
                                         0 & \hbox{if } i_t = i_{t+1}, \\
                                         1_{\mathbf{i}} & \hbox{if } |i_t - i_{t+1}| > 1 , \\
                                         ( x_t + x_{t+1} )1_\mathbf{i} & \hbox{if } |i_t - i_{t+1}| = 1,
                                       \end{array}
                                     \right. \\
& (\tau_t \tau_{t+1} \tau_t - \tau_{t+1} \tau_t \tau_{t+1} ) 1_\mathbf{i} = \left\{
                                                                                   \begin{array}{ll}
                                                                                     1_\mathbf{i} & \hbox{if } i_t = i_{t+2} \text{ and } |i_t - i_{t+1}|=1, \\
                                                                                     0 & \hbox{otherwise},
                                                                                   \end{array}
                                                                                 \right. \\
& (\tau_t x_k - x_{\tau_t(k)} \tau_t )1_\mathbf{i} = \left\{
                                                           \begin{array}{ll}
                                                             1_\mathbf{i} & \hbox{if } k=t \text{ and } i_t = i_{t+1}, \\
                                                             -1_\mathbf{i} & \hbox{if } k = t+1 \text{ and } i_t = i_{t+1},  \\
                                                             0 & \hbox{otherwise.}
                                                           \end{array}
                                                         \right.
\end{aligned}
\end{equation}

\end{Def}
For simplicity, we set $R(0) = \C$. The grading on $R(\alpha)$ is given by
$$ \deg(1_\mathbf{i}) =0, \quad \deg(x_k 1_\mathbf{i})= 2, \quad \deg(\tau_t 1_\mathbf{i})= -a_{i_t,i_{t+1}}. $$
For $\lambda = \sum_{i=1}^n a_i \varpi_i \in P^+ $, let
$I^{\lambda}(\alpha)$ be the two-side ideal of $R(\alpha)$ generated
by $x_1^{a_{i_1}} 1_\mathbf{i}\ (\mathbf{i}=(i_1, \ldots, i_d) \in
I^\alpha) $, and define
$$ R^\lambda(\alpha) := R(\alpha) / I^\lambda(\alpha). $$
The algebra $R^{\lambda}(\alpha)$ is called the {\it cyclotomic
quotient} of $R(\alpha)$ at $\lambda$.

Let $R(\alpha)$-$\fMod$ (resp.\ $R^\lambda(\alpha)$-$\fMod$) be the
category of finite dimensional graded $R(\alpha)$-modules (resp.\
$R^\lambda(\alpha)$-modules). For any irreducible graded module
$M \in R^\lambda(\alpha)$-$\fMod$, $M$ can be
viewed as an irreducible graded $R(\alpha)$-module annihilated by
$I^\lambda(\alpha)$, which defines a functor
$${\infl_\lambda}:
R^\lambda(\alpha) \text{-}\fMod \to  R(\alpha) \text{-}\fMod.$$
For $M \in R^\lambda(\alpha)$-$\fMod$, $\infl_\lambda
M$ is called the {\it inflation} of $M$. On the
other hand, from the natural projection $R(\alpha) \to
R^\lambda(\alpha)$, we define the functor $\pr_\lambda: R(\alpha)
\text{-}\fMod \to  R^\lambda(\alpha) \text{-}\fMod$ by
$$\pr_\lambda N := N / I^\lambda(\alpha) N \quad  \text{ for }  N \in R(\alpha) \text{-}\fMod.$$
From now on, when there is no danger of confusion, we identify any
irreducible graded $R^\lambda(\alpha)$-module with an irreducible
graded $R(\alpha)$-module annihilated by $I^\lambda(\alpha)$ via the
funtor $\infl_\lambda$.

The algebra $R(\alpha)$ has a graded anti-involution
\begin{align} \label{Eq: anti-involution}
 \psi : R(\alpha) \longrightarrow R(\alpha)
\end{align}
which is the identity on generators. Using this anti-involution, for any finite dimensional
graded $R(\alpha)$-module $M$, the dual space $M^* := \Hom_{\C}( M, \C )$ of $M$ has the $R(\alpha)$-module structure given by
$$ (r \cdot f)(m):= f(\psi(r)m) \quad (r \in R(\alpha),  m\in M).  $$
Note that, if $M$ is irreducible, then $M \simeq M^*$ by \cite[Theorem 3.17]{KL09}.

Given $M = \bigoplus_{i \in \Z} M_i$, let $M\langle k \rangle$ denote the graded
module obtained from $M$ by shifting the grading by $k$; i.e.,
$$M \langle k \rangle := \bigoplus_{i \in \Z} M \langle k \rangle _i,$$
where $M \langle k \rangle_i := M_{i-k} $ for $i \in \Z $. We define the {\it
$q$-dimension} $\qdim(M)$ of $M = \bigoplus_{i \in \Z} M_i$ to be
$$ \qdim(M) := \sum_{i \in \Z} (\dim M_i) q^i.  $$
Set
\begin{align*}
R &:= \bigoplus_{\alpha \in Q^+} R(\alpha), \qquad K_0(R) := \bigoplus_{\alpha \in Q^+} K_0( \text{$R(\alpha)$-$\fMod$} ),\\
R^\lambda &:= \bigoplus_{\alpha \in Q^+} R^\lambda(\alpha) , \qquad
K_0(R^\lambda) := \bigoplus_{\alpha \in Q^+} K_0( \text{$R^\lambda(\alpha)$-$\fMod$} ),
\end{align*}
where $K_0( \text{$R(\alpha)$-$\fMod$} )$ (resp.\ $K_0(
\text{$R^\lambda(\alpha)$-$\fMod$} )$) is the Grothendieck group of
$R(\alpha)$-$\fMod$ (resp.\ $R^\lambda(\alpha)$-$\fMod$). For $M \in
\text{$R(\alpha)$-$\fMod$}$  (resp.\ $\text{$R^\lambda(\alpha)$-$\fMod$} $), we denote by $[M]$ the
isomorphism class of $M$ in $K_0( \text{$R(\alpha)$-$\fMod$} )$
(resp.\ $K_0( \text{$R^\lambda(\alpha)$-$\fMod$} )$).
 Then
$K_0(R)$ (resp.\ $K_0(R^\lambda)$) has the $\Z[q,q^{-1}]$-module structure given by $q[M] = [M\langle 1 \rangle]$.

Define the {\it $q$-character} $\qch(M)$ (resp.\ {\it character}
$\ch(M)$) of $M \in R(\alpha)$-$\fMod$ by
$$ \qch(M) := \sum_{\mathbf{i} \in I^\alpha} \qdim(1_\mathbf{i} M ) \ \mathbf{i} \qquad \text{(resp.\ }
 \ch(M) := \sum_{\mathbf{i} \in I^\alpha} \dim(1_\mathbf{i} M ) \ \mathbf{i} ). $$
Note that the evaluation of $\qdim(1_\mathbf{i} M )$ at $q=1$ is $\dim(1_\mathbf{i} M )$. For $M \in R(\alpha)$-$\fMod$ and $N \in R(\beta)$-$\fMod$, we set
\begin{align*}
\qch(M) * \qch(N) &:= \sum_{\mathbf{i} \in I^\alpha, \mathbf{j} \in I^\beta} \qdim(1_\mathbf{i} M )  \qdim(1_\mathbf{j} N ) \ \mathbf{i} * \mathbf{j}, \\
\ch(M) * \ch(N) &:= \sum_{\mathbf{i} \in I^\alpha, \mathbf{j} \in I^\beta} \dim(1_\mathbf{i} M )  \dim(1_\mathbf{j} N ) \ \mathbf{i} * \mathbf{j}.
\end{align*}

For $M,N \in R(\alpha)$-$\fMod$, let $\Hom(M,N)$ be the $\C$-vector space of degree preserving homomorphisms,
and $\Hom(M \langle k \rangle,N) = \Hom(M,N \langle -k \rangle)$ be the $\C$-vector space of homogeneous homomorphisms of degree $k$. Define
$$ \HOM(M,N):= \bigoplus_{k\in \Z} \Hom(M,N \langle k \rangle). $$

Let $\beta_1, \ldots, \beta_k \in Q^+$ and set $\beta = \beta_1 +
\cdots + \beta_k$. Then there is a natural embedding
$$ \iota_{\beta_1, \ldots   ,\beta_k}: R(\beta_1) \otimes \cdots \otimes R(\beta_n) \hookrightarrow R(\beta),  $$
which yields the following functors from $R(\beta_1) \otimes \cdots \otimes R(\beta_n)$-$\fMod$ to $R(\beta)$-$\fMod$:
\begin{align*}
\ind_{\beta_1, \ldots   ,\beta_k}\ - &:= R(\beta) \otimes_{R(\beta_1) \otimes \cdots \otimes R(\beta_n)}\  - ,\\
\coind_{\beta_1, \ldots   ,\beta_k}\ - &:= \HOM_{R(\beta_1) \otimes \cdots \otimes R(\beta_n)}(R(\beta), \ -).
\end{align*}

The properties of the functors $\ind, \coind$ and $\res$ are
summarized in the following lemmas.

\begin{Lem}\cite[(2.3)]{LV09} \label{Lem: Frobenius reciprocity}
For $M \in R(\beta_1) \otimes \cdots \otimes R(\beta_k)$-$\fMod$ and
$N \in R(\beta)$-$\fMod$, we have
\begin{align*}
\HOM_{R(\beta)} ( \ind_{\beta_1, \ldots \beta_k}M, N) &\cong \HOM_{R(\beta_1)\otimes \cdots \otimes R(\beta_k)} ( M, \res_{\beta_1,\ldots,\beta_k}N),\\
\HOM_{R(\beta)} (N, \coind_{\beta_1, \ldots, \beta_k} M ) &\cong \HOM_{R(\beta_1) \otimes \cdots \otimes R(\beta_k)}(\res_{\beta_1, \ldots ,\beta_k}N, M).
\end{align*}
\end{Lem}

\begin{Thm}\cite[Theorem 2.2]{LV09} \label{Thm: ind and coind}
Let $M_i \in R(\beta_i)$-$\fMod$ $(i=1, \ldots, k)$ and
$$K := -\sum_{i > j} (\beta_i|\beta_j), $$
where $( \ | \ )$ is the nondegenerate symmetric bilinear form on
$Q$ defined by $(\alpha_i | \alpha_j) = a_{ij}$ $(i,j \in I)$.  Then
there exists a homogeneous isomorphism
$$ \ind_{\beta_1, \ldots, \beta_k}M_1\boxtimes \cdots \boxtimes M_k \cong \coind_{\beta_k, \ldots, \beta_1} (M_k \boxtimes \cdots \boxtimes M_1)\langle K \rangle.$$
\end{Thm}
When there is no ambiguity, we will write $\res, \ind, \coind $ for
$ \res_{\beta_1, \ldots ,\beta_k}, \ind_{\beta_1, \ldots \beta_k} $
and $ \coind_{\beta_1, \ldots, \beta_k}$, respectively.

We first consider the special case when $\alpha = m \alpha_i$. It is
known that $R(m\alpha_i)$ is isomorphic to the nilHecke ring $NH_m$
(see \cite[Example 2.2]{KL09}). Thus $R(m\alpha_i)$ has only one
irreducible representation
\begin{equation} \label{Eq: L(i^m)}
L(i^m) \cong \ind_{\C[x_1, \ldots, x_m]}\ \mathbf{1}
\end{equation}
up to grading shift, where $\mathbf{1}$ is the 1-dimensional trivial
module over $\C[x_1, \ldots, x_m]$. We define 
$ \qch(\mathbf{1}) :=  (i,\ldots,i) $. 
%$ \qch(\mathbf{1}) := q^{m(m-1)/2}\cdot (i,\ldots,i) $. 
Since $\dim L(i^m)=m!$, for any $M
\in R(\alpha)$-$\fMod$ and $\mathbf{i} = ( \cdots, \underbrace{i,
\dots, i}_{m}, \cdots) \in I^\alpha$ with $\dim( 1_\mathbf{i} M ) >
0$, we have
\begin{align} \label{Eq: L(i^m):dimension}
\dim( 1_\mathbf{i} M ) \ge m! .
\end{align}

Take a nonzero element $\zeta$ in $\mathbf{1}$. Then $L(i^m)$ is
generated by $1\otimes \zeta$ and, by \cite[Theorem 2.5]{KL09},
$L(i^m)$ has a basis $\{ w \cdot 1\otimes \zeta |\ w \in
\Sigma_m\}$. Set $L^{0} = \{0\}$ and
$$ L^k := \{ v \in L(i^m)|\ x_m^k \cdot v = 0  \} \quad (k=1,2,\ldots,m). $$
Since $x_m$ commutes with all $x_i\ (i=1,\ldots,m-1)$ and $\tau_j\ (j=1,\ldots,m-2)$,
$L^k$ can be considered as $R((m-1)\alpha_i)$-module. Moreover, by a direct computation, we have
\begin{align} \label{Eq: filteration of L(i^m)}
 L^k = \{ w \tau_{m-1}\cdots \tau_{m-k+1} \cdot 1 \otimes \zeta|\ w \in \Sigma_{m-1} \}.
\end{align}
It follows that $L^k / L^{k-1}$ is isomorphic to $L(i^{m-1})$ for each $k=1, \ldots, m$.

We now return to the general case.  Let $M $ be a finite dimensional
graded $R(\alpha)$-module. For any $\beta \in Q^{+}$, set $1_{\beta}
:= \sum_{\mathbf{i} \in I^{\beta}} 1_\mathbf{i}$. For $i \in I$,
define
\begin{equation}\label{Eq:e_i}
\begin{aligned}
\Delta_{i^k} M &:= 1_{\alpha-k\alpha_i} \otimes 1_{k\alpha_i} M, \\
e_i M &:= \res_{\alpha-\alpha_i}^{\alpha-\alpha_i, \alpha_i} \circ \Delta_i M  .
\end{aligned}
\end{equation}
Then $e_i$ may be considered as a functor: $ K_0(R(\alpha)\text{-}\fMod) \to K_0(R(\alpha - \alpha_i)\text{-}\fMod) $.

\begin{Lem} \label{Lem: exact}
Let $M \in R(\alpha)\text{-}\fMod$ and $ N \in R(\beta)\text{-}\fMod$.
Then we have the following exact sequence:
$$ 0 \to \ind_{\alpha, \beta-\alpha_i}M \boxtimes e_i N \to e_i(\ind_{\alpha,\beta}M\boxtimes N)
\to \ind_{\alpha-\alpha_i, \beta} e_i M\boxtimes N \langle -\beta(h_i) \rangle \to 0.  $$
\end{Lem}
\begin{proof}
Our assertion follows from the Khovanov-Lauda-Rouquier algebra version
of the Mackey's theorem \cite[Proposition 2.18.]{KL09}.
\end{proof}

\begin{Prop}\cite[Corollary 2.15]{KL09} \label{Prop: Serre}
For any finitely-generated graded $R(\alpha)$-module $M$, we have
$$ \sum_{r=0}^{1-a_{ij}} (-1)^r e_i^{(1-a_{ij}-r)} e_j e_i^{(r)}[M]=0,$$
where $e_i^{(r)}[M] = \frac{1}{[r]_q !} \ [e_i^r M]$ for $i \in I $ and $r \in \Z_{\ge0}$.
\end{Prop}

Let us reinterpret the quantum Serre relations given in  Proposition
\ref{Prop: Serre}. Let $M$ be a finite-dimensional graded
$R(\alpha)$-module. Consider the sequences $\mathbf{i}_{\{i,j
\}},\mathbf{i}_{\{j,i \}} \in I^{\alpha}$ of the form:
$$ \mathbf{i}_{\{i,j \}} := \mathbf{k}_1 * (i,j) * \mathbf{k}_2, \quad
\mathbf{i}_{\{ j,i \}} := \mathbf{k}_1 * (j,i) * \mathbf{k}_2,  $$
where $\mathbf{k}_1, \mathbf{k}_2$ are sequences satisfying
$\mathbf{k}_1
* \mathbf{k}_2 \in I^{\alpha - \alpha_i - \alpha_j}$. Suppose $|i-j|
> 1$. It follows from Proposition  \ref{Prop: Serre} that
\begin{align} \label{Eq: Serre |i-j|>1}
\dim (1_{\mathbf{i}_{\{i,j \}}} M) =  \dim (1_{\mathbf{i}_{\{j,i \}}}  M ).
\end{align}
We now consider the case $|i-j| = 1$.  Let
\begin{align*}
 \mathbf{i}_{\{i\pm 1,i,i \}} &:= \mathbf{k}_1 * (i\pm 1,i,i) * \mathbf{k}_2, \\
\mathbf{i}_{\{i,i\pm 1,i \}} &:= \mathbf{k}_1 * (i,i\pm 1,i) * \mathbf{k}_2, \\
\mathbf{i}_{\{i,i,i\pm 1 \}} &:= \mathbf{k}_1 * (i,i,i \pm 1) * \mathbf{k}_2
\end{align*}
for some sequences $\mathbf{k}_1, \mathbf{k}_2$ with $\mathbf{k}_1 *
\mathbf{k}_2 \in I^{\alpha - 2\alpha_i - \alpha_{i\pm 1}}$. Then,
from Proposition  \ref{Prop: Serre}, we have
\begin{align} \label{Eq: Serre |i-j|=1}
2 \dim (1_{\mathbf{i}_{\{i,i \pm 1, i \}}} M) =  \dim (1_{\mathbf{i}_{\{i\pm 1,i,i \}}}  M )
+ \dim (1_{\mathbf{i}_{\{i,i,i\pm 1 \}}}  M ).
\end{align}

Let $\mathfrak{B}(\infty)$ denote the set of isomorphism classes of
irreducible graded $R$-modules, and define
\begin{align*}
\wt(M) &:= -\alpha, \\
\e_i M &:= \soc\ e_i M, \\
\f_i M &:= \hd\ \ind_{\alpha, \alpha_i} M \boxtimes L(i), \\
\varepsilon_i(M) &:= \max\{ k \ge 0|\ \e_i^k M \ne 0 \} \\
\varphi_i(M) &:= \varepsilon_i(M) + \langle h_i, \wt(M) \rangle.
\end{align*}

\begin{Thm}\cite[Theorem 7.4]{LV09} \label{Thm: iso-LV B(infty)}
The sextuple $(\mathfrak{B}(\infty),\wt, \e_i, \f_i, \varepsilon_i,
\varphi_i)$ becomes a crystal, which is isomorphic to the crystal
$B(\infty)$ of $U_q^{-}(\slm_{n+1})$.
\end{Thm}

For $M \in R^\lambda(\alpha)$-$\fMod$ and $N \in R(\alpha)$-$\fMod$,
let $\infl_\lambda M$ be the inflation of $M$,
and $\pr_\lambda N $ be the quotient of $ N $ by $ I^\lambda(\alpha) N$.
Let $\mathfrak{B}(\lambda)$ denote the set of isomorphism classes of irreducible $R^\lambda$-modules, and
for $M \in R^\lambda(\alpha)$-$\fMod$, define
\begin{align*}
\wt^\lambda(M) &:= \lambda - \alpha, \\
\e_i^\lambda M &:=  \pr_\lambda \circ \e_i \circ \infl_\lambda M, \\
\f_i^\lambda M &:= \pr_\lambda \circ \f_i \circ \infl_\lambda M, \\
\varepsilon_i^\lambda(M) &:= \max\{ k \ge 0|\ (\e_i^\lambda)^k M \ne 0 \}, \\
\varphi_i^\lambda(M) &:= \varepsilon_i^\lambda(M) + \langle h_i, \wt^\lambda(M) \rangle.
\end{align*}

\begin{Thm}\cite[Theorem 7.5]{LV09} \label{Thm: iso-LV}
The sextuple $(\mathfrak{B}(\lambda),\wt^\lambda, \e_i^\lambda,
\f_i^\lambda, \varepsilon_i^\lambda, \varphi_i^\lambda)$ becomes a
crystal, which is isomorphic to the crystal $B(\lambda)$ of the
irreducible highest weight $U_q(\slm_{n+1})$-module $V(\lambda)$.
\end{Thm}

The following lemma is an analogue of \cite[Theorem 5.5.1]{K05}.

%and gives a connection between two functors $e_i$ and $\e_i$.

\begin{Lem} \label{Lem: e_i and tilde e_i}
Let $M$ be an irreducible $R(\alpha)$-module. Set $\varepsilon :=
\varepsilon_i(M)$. Then we have

\begin{enumerate}
\item
$$[e_i M] = q^{-\varepsilon+1}[\varepsilon]_q[\e_i M] + \sum_k c_k [N_k],$$
 where $N_k$ are irreducible modules with
$ \varepsilon_i(N_k) < \varepsilon_i(\e_i M) = \varepsilon-1 $,
\item
$$ [e_i^{\varepsilon}M] = q^{-\frac{\varepsilon(\varepsilon-1)}{2}}[\varepsilon]_q ![\e_i^{\varepsilon}M]. $$
\end{enumerate}
\end{Lem}

\begin{proof}
Since the assertion (2) follows from the assertion (1) immediately, it suffices to prove (1).
By in \cite[Lemma 3.8]{KL09},
$$\Delta_{i^\varepsilon} M \cong N \boxtimes L(i^\varepsilon)$$
for some irreducible $N \in R(\alpha - \varepsilon \alpha_i)$-$\fMod$ with $\varepsilon_i(N) = 0$. Then we have
$$  N \boxtimes L(i^\epsilon) \buildrel \sim \over \longrightarrow \Delta_{i^\epsilon} M \subset  \res_{\alpha - \varepsilon \alpha_i, \varepsilon \alpha_i} M, $$
which yields
$$ 0 \longrightarrow K \longrightarrow \ind_{\alpha - \varepsilon \alpha_i, \varepsilon \alpha_i} N \boxtimes L(i^\epsilon)  \longrightarrow
M \longrightarrow 0$$
for some $R(\alpha)$-module $K$. Note that $\varepsilon_i(K) < \varepsilon$.

On the other hand, it follows from
$\eqref{Eq: L(i^m)}$ and $\eqref{Eq: filteration of L(i^m)}$ that
\begin{align*}
[\Delta_i L(i^\varepsilon)] = q^{-\varepsilon+1}[\varepsilon]_q [L(i^{\varepsilon-1}) \boxtimes L(i)].
\end{align*}
Since $\varepsilon_i(N) = 0$, it follows from \cite[Proposition 2.18]{KL09} that
$$[\Delta_i \ind_{\alpha - \varepsilon \alpha_i, \varepsilon \alpha_i} N \boxtimes L(i^\epsilon)]
 = q^{-\varepsilon+1}[\varepsilon]_q [\ind_{\alpha-\varepsilon\alpha_i, (\varepsilon-1)\alpha_i, \alpha_i }^{\alpha - \alpha_i, \alpha_i}
N \boxtimes L(i^{\varepsilon-1}) \boxtimes L(i)].$$
By \cite[Lemma 3.9]{KL09} and \cite[Lemma 3.13]{KL09}, we obtain
$$ \hd  (\ind_{\alpha-\varepsilon\alpha_i, (\varepsilon-1)\alpha_i, \alpha_i }^{\alpha - \alpha_i, \alpha_i}
N \boxtimes L(i^{\varepsilon-1}) \boxtimes L(i) )  \cong
(\f_i^{\varepsilon-1}N) \boxtimes L(i) \cong \e_i M \boxtimes L(i) ,
$$ and all the other composition factors of $
\ind_{\alpha-\varepsilon\alpha_i, (\varepsilon-1)\alpha_i, \alpha_i
}^{\alpha - \alpha_i, \alpha_i} N \boxtimes L(i^{\varepsilon-1})
\boxtimes L(i)$ are of the form $L \boxtimes L(i)$ with
$\varepsilon_i(L) < \varepsilon-1$. Moreover, since
$\varepsilon_i(K) < \varepsilon$, all composition factors of
$\Delta_i(K)$ are of the form $L\boxtimes L(i)$ with
$\varepsilon_i(L) < \varepsilon-1$. Therefore, we obtain
$$[e_i M] =
q^{-\varepsilon+1}[\varepsilon]_q[\e_i M] + \sum_k c_k [N_k],$$ where $N_k$ are
irreducible modules with $ \varepsilon_i(N_k) < \varepsilon_i(\e_i
M) = \varepsilon-1 $.
\end{proof}

The following lemmas are analogues of \cite[Proposition 8, Proposition 9]{V02}, which will play crucial roles
in proving our main theorem.

\begin{Lem} \label{Lem: char}
For $\beta_1, \ldots, \beta_k \in Q^+$, let $\gamma_i$ be a 1-dimensional graded
$R(\beta_i)$-module.
\begin{enumerate}
\item If $Q$ is any graded quotient of $\ind_{\beta_1, \cdots, \beta_k } \gamma_1\boxtimes \cdots \boxtimes \gamma_k$,
then $\ch Q$ contains
$$\ch(\gamma_1) \ast \cdots \ast \ch(\gamma_k) .$$
\item If $L$ is any graded submodule of $\ind_{\beta_1, \cdots, \beta_k } \gamma_1\boxtimes \cdots \boxtimes \gamma_k$,
then $\ch L$ contains
$$\ch(\gamma_k) \ast \cdots \ast \ch(\gamma_1) .$$
\end{enumerate}
\end{Lem}
\begin{proof}

It follows from Lemma \ref{Lem: Frobenius reciprocity} that
$$ \HOM_{R(\beta_1)\otimes \cdots \otimes  R(\beta_k)} (  \gamma_1\boxtimes \cdots \boxtimes \gamma_k,\
\res_{\beta_1, \ldots, \beta_k} Q )$$
is nontrivial, which implies that $\ch Q$ contains the concatenation
$\ch(\gamma_1) \ast \cdots \ast \ch(\gamma_k) $.

Consider now the assertion (2). By Lemma \ref{Lem: Frobenius reciprocity} and Theorem \ref{Thm: ind and coind}, we have
\begin{align*}
& \quad  \HOM_{R(\beta_1 + \cdots + \beta_k)}(L,\  \ind_{\beta_1, \cdots, \beta_k } \gamma_1\boxtimes \cdots \boxtimes \gamma_k)\\
&\cong
\HOM_{R(\beta_1 + \cdots + \beta_k)}(L,\  \coind_{\beta_k, \cdots, \beta_1 } \gamma_k\boxtimes \cdots \boxtimes \gamma_1)\\
& \cong
\HOM_{R(\beta_k) \otimes \cdots \otimes R(\beta_1)}(\res_{\beta_k, \ldots, \beta_1} L,\
\gamma_k\boxtimes \cdots \boxtimes \gamma_1).
\end{align*}
Since the above spaces are non-trivial, the assertion (2) follows.
\end{proof}

\begin{Lem} \label{Lem: irreduciblity}
Let $\beta_1, \ldots, \beta_k \in Q^+$ and let $M$ be an irreducible
$R(\beta_1) \otimes \cdots \otimes R(\beta_k)$-module. Assume that
$\mathbf{i} \in I^{\beta_1+\cdots \beta_k}$ appears in
$\ch(\ind_{\beta_1, \ldots, \beta_k}M)$ with coefficient $m$.
\begin{enumerate}
\item  Suppose that $\mathbf{i}$ occurs with coefficient $m$
in the character of any submodule of $\ind_{\beta_1, \ldots, \beta_k}M$.
Then $\soc \ind_{\beta_1, \ldots, \beta_k}M $ is irreducible and occurs
with multiplicity one as a composition factor of $\ind_{\beta_1, \ldots, \beta_k}M $.
If $\mathbf{i}$ occurs in $\ch(\hd \ind_{\beta_1, \ldots, \beta_k}M )$,
then $\ind_{\beta_1, \ldots, \beta_k}M$ is irreducible.
\item  Suppose that $\mathbf{i}$ occurs with coefficient $m$
in the character of any quotient of $\ind_{\beta_1, \ldots, \beta_k}M$.
Then $\hd \ind_{\beta_1, \ldots, \beta_k}M $ is irreducible and occurs
with multiplicity one as a composition factor of $\ind_{\beta_1, \ldots, \beta_k}M $.
If $\mathbf{i}$ occurs in $\ch(\soc \ind_{\beta_1, \ldots, \beta_k}M )$,
then $\ind_{\beta_1, \ldots, \beta_k}M$ is irreducible.
\end{enumerate}
\end{Lem}
\begin{proof}
Let $L$ be a component of $\soc \ind_{\beta_1, \ldots, \beta_k}M $.
By hypothesis, $\ch L$ contains $\mathbf{i}$ with coefficient $m$ as
a term. Since $\mathbf{i}$ occurs with coefficient $m$ in the
character of any submodule, $\soc \ind_{\beta_1, \ldots, \beta_k}M $
should be irreducible. In a similar manner, one can show that $\soc
\ind_{\beta_1, \ldots, \beta_k}M $ occurs with multiplicity one as a
composition factor of $\ind_{\beta_1, \ldots, \beta_k}M $. Suppose
that $\mathbf{i}$ occurs in $\ch(\hd \ind_{\beta_1, \ldots,
\beta_k}M )$. Since the multiplicity of $\mathbf{i}$ in $\soc
\ind_{\beta_1, \ldots, \beta_k}M $ is equal to the multiplicity of
$\mathbf{i}$ in $ \ind_{\beta_1, \ldots, \beta_k}M $,
$\ind_{\beta_1, \ldots, \beta_k}M$ should be irreducible.

The assertion (2) can be proved in a similar manner.
\end{proof}

\vskip 3em

\section{Irreducible $R^{\lambda}$-modules and semistandard tableaux} \label{Sec: KL alg and B(lambda)}

In this section, we prove the main results of our paper. We give an
explicit construction of irreducible graded $R^{\lambda}$-modules
(Theorem \ref{Thm: hd indT is irr}) and show that there exists an
explicit crystal isomorphism $\Phi_{\lambda}: \mathbf{B}(\lambda)
\rightarrow \mathfrak{B}(\lambda)$ (Theorem \ref{Thm: crystal iso}). 
From now on, isomorphisms of modules are allowed to be homogeneous.

For $a, \ell \in \Z_{>0}$ with $a + \ell-1 \le n $,
let
\begin{align*}
\alpha_{(a;\ell)} &:=  \alpha_{a} + \alpha_{a+1} + \cdots + \alpha_{a + \ell-1} \in Q^+, \\
\mathbf{i}_{(a;\ell)} &:= (a, a+1, \ldots, a+\ell-1) \in I^{\alpha_{(a;\ell)}}.
\end{align*}
Define $ \ms_{(a;\ell)}$
to be the 1-dimensional $R(\alpha_{(a;\ell)})$-module $\C v$ given by
\begin{align} \label{Eq: def of ms}
x_i v = 0,\quad \tau_j v = 0, \quad  1_{\mathbf{i}} v = \left\{
                                                           \begin{array}{ll}
                                                             v & \hbox{ if } \mathbf{i} = \mathbf{i}_{(a;\ell)}, \\
                                                             0 & \hbox{otherwise}.
                                                           \end{array}
                                                         \right.
\end{align}
%for any nonzero element $v\in \ms_{(a;\ell)}$.
The module $\ms_{(a;\ell)}$ can be visualized as follows: \vskip 1em
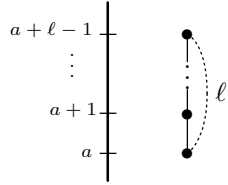
\begin{center}
\begin{texdraw}
\drawdim em \setunitscale 0.15  \linewd 0.6
\arrowheadtype t:F

\move(10 0) \lvec(10 -45)  % \move(0 -60) \lvec(200 -60)
\linewd 0.3
\move(8 -38) \lvec(12 -38)  \htext(3 -39){\scriptsize$a$}
\move(8 -28) \lvec(12 -28)  \htext(-5 -29){\scriptsize$a+1$}
\move(8 -18)   \htext(0 -19){\scriptsize$\vdots$}
\move(8 -8) \lvec(12 -8)  \htext(-15 -9){\scriptsize$a+\ell-1$}

\move(30 -8)
\lvec(30 -15)
\move(30 -22)
\lvec(30 -38)
\htext(28.1 -10){$\bullet$}
\htext(28.9 -22){$\vdots$}
\htext(28.1 -30){$\bullet$}
\htext(28.1 -40){$\bullet$}

\lpatt(0.5 1)
\move(30 -8)\clvec(30 -8)(35 -10)(35 -23)
\move(30 -38)\clvec(30 -38)(35 -35)(35 -23)
\htext(37 -25){$\ell$}

\end{texdraw}
\end{center}
For simplicity, set $\ms_{(a; 0)} := \C$. Note that $\ms_{(a;\ell)}$ is graded and $\ch \ms_{(a;\ell)} = \mathbf{i}_{(a;\ell)}$.

%\begin{Def} \label{Def: 1-dim rep from YD}
Let $\mu = (\mu_1\ge \cdots\ge \mu_r > 0)$ be a Young diagram, and $k \in \Z_{>0}$.
Suppose that $  k + \mu_1  - 1 \le n $. Define
\begin{equation} \label{Eq: 1-dim rep from YD}
\begin{aligned}
\alpha_{\mu}[k] & := \alpha_{(k;\mu_1)} + \cdots +
\alpha_{(k;\mu_r)} \in Q^+,\\
 \ms_{\mu}[k] &:=  \ms_{(k;\mu_1)} \boxtimes \ms_{(k;\mu_2)}
\boxtimes \cdots \boxtimes \ms_{(k;\mu_r)}, \\
\tms_{\mu}[k] &:=  \ms_{(k;\mu_r)} \boxtimes \ms_{(k;\mu_{r-1})}
\boxtimes \cdots \boxtimes \ms_{(k;\mu_1)}.
\end{aligned}
\end{equation}

%\end{Def}

Pictorially, the modules $ \ms_{\mu}[k]$ and $\tms_{\mu}[k]$ may be
viewed as follows: \vskip 1em
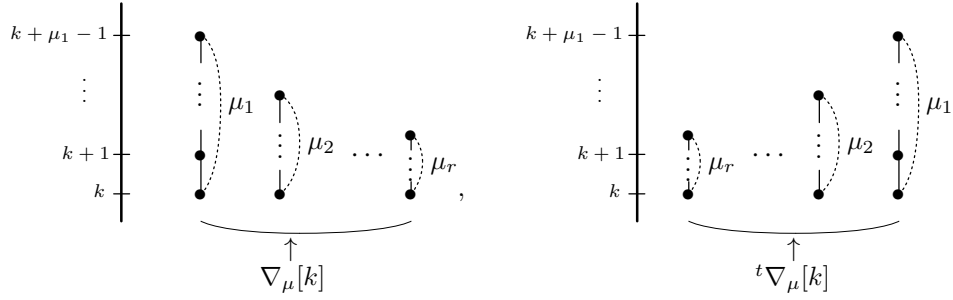
\begin{center}
\begin{texdraw}
\drawdim em \setunitscale 0.15  \linewd 0.6
\arrowheadtype t:F

\move(10 0) \lvec(10 -55)
\linewd 0.3
\move(8 -48) \lvec(12 -48)  \htext(3 -49){\scriptsize$k$}
\move(8 -38) \lvec(12 -38)  \htext(-5 -40){\scriptsize$k+1$}
 \htext(0 -25){\scriptsize$\vdots$}
\move(8 -8) \lvec(12 -8)  \htext(-18 -10){\scriptsize$k+\mu_1-1$}

%mu_1
\move(30 -8)
\lvec(30 -15)
\move(30 -32)
\lvec(30 -48)
\htext(28.3 -10){$\bullet$}
\htext(28.9 -26){$\vdots$}
\htext(28.3 -40){$\bullet$}
\htext(28.3 -50){$\bullet$}

%mu_2

\move(50 -23)
\lvec(50 -30)
\move(50 -42)
\lvec(50 -48)
\htext(48.3 -25){$\bullet$}
\htext(49.1 -39){$\vdots$}
\htext(48.3 -50){$\bullet$}

\htext(68 -40){\large$\cdots$}

%mu_r

\move(83 -33)
\lvec(83 -37)
\move(83 -45)
\lvec(83 -48)
\htext(81.3 -35){$\bullet$}
\htext(82.1 -45){$\vdots$}
\htext(81.3 -50){$\bullet$}

\htext(95 -50){,}

\linewd 0.2
\move(30 -55)\clvec(30 -55)(35 -58)(55 -58)
\move(83 -55)\clvec(83 -55)(78 -58)(55 -58)
\htext(51 -65){$\uparrow$}
\htext(45 -73){$\ms_{\mu}[k]$}
\linewd 0.3

%-----------------------------

\linewd 0.6
\move(140 0) \lvec(140 -55)
\linewd 0.3
\move(138 -48) \lvec(142 -48)  \htext(133 -49){\scriptsize$k$}
\move(138 -38) \lvec(142 -38)  \htext(125 -40){\scriptsize$k+1$}
 \htext(130 -25){\scriptsize$\vdots$}
\move(138 -8) \lvec(142 -8)  \htext(112 -10){\scriptsize$k+\mu_1-1$}

%mu_r

\move(153 -33)
\lvec(153 -37)
\move(153 -45)
\lvec(153 -48)
\htext(151.3 -35){$\bullet$}
\htext(152.1 -45){$\vdots$}
\htext(151.3 -50){$\bullet$}

\htext(169 -40){\large$\cdots$}

%mu_2

\move(186 -23)
\lvec(186 -30)
\move(186 -42)
\lvec(186 -48)
\htext(184.3 -25){$\bullet$}
\htext(185.1 -39){$\vdots$}
\htext(184.3 -50){$\bullet$}

%mu_1
\move(206 -8)
\lvec(206 -15)
\move(206 -32)
\lvec(206 -48)
\htext(204.3 -10){$\bullet$}
\htext(204.9 -26){$\vdots$}
\htext(204.3 -40){$\bullet$}
\htext(204.3 -50){$\bullet$}

\linewd 0.2
\move(153 -55)\clvec(153 -55)(158 -58)(180 -58)
\move(206 -55)\clvec(206 -55)(201 -58)(180 -58)
\htext(178 -65){$\uparrow$}
\htext(170 -73){$\tms_{\mu}[k]$}
\linewd 0.3

%-----------------------------

%mu_1
\lpatt(0.5 1)
\move(30 -8)\clvec(30 -8)(35 -10)(35 -25)
\move(30 -48)\clvec(30 -48)(35 -45)(35 -25)
\htext(37 -28){$\mu_1$}

%mu_2
\move(50 -23)\clvec(50 -23)(55 -25)(55 -35)
\move(50 -48)\clvec(50 -48)(55 -45)(55 -35)
\htext(57 -38){$\mu_2$}

%mu_k
\move(83 -33)\clvec(83 -33)(86 -35)(86 -40)
\move(83 -48)\clvec(83 -48)(86 -45)(86 -40)
\htext(88 -43){$\mu_r$}

%-----------------------------

%mu_k
\move(153 -33)\clvec(153 -33)(156 -35)(156 -40)
\move(153 -48)\clvec(153 -48)(156 -45)(156 -40)
\htext(158 -43){$\mu_r$}

%mu_2
\move(186 -23)\clvec(186 -23)(191 -25)(191 -35)
\move(186 -48)\clvec(186 -48)(191 -45)(191 -35)
\htext(193 -38){$\mu_2$}

%mu_1
\move(206 -8)\clvec(206 -8)(211 -10)(211 -25)
\move(206 -48)\clvec(206 -48)(211 -45)(211 -25)
\htext(213 -28){$\mu_1$}

\end{texdraw}
\end{center}
\vskip 0.5em

One of the key ingredients of the proof of Theorem \ref{Thm: crystal
iso} is the fact that $\ind \ms_{\mu}[k]$ is irreducible for any
Young diagram $\mu$ and $k\in \Z_{>0}$. To prove this, we need
several lemmas. The following lemma may be obtained by translating
the {\it linking rule} given in \cite[Lemma 4]{V02} into the
language of Khovanov-Lauda-Rouquier algebras.

\begin{Lem} \label{Lem: linking}
 Let $a_i, \ell_i \in \Z_{>0}$ with $a_i + \ell_i-1 \le n \ (i=1,2)$.
\begin{enumerate}
\item If $a_1+\ell_1 - 1 < a_2$,
then
$$ \ind \ms_{(a_1;\ell_1)} \boxtimes \ms_{(a_2;\ell_2)} \cong \ind \ms_{(a_2;\ell_2)}
\boxtimes \ms_{(a_1;\ell_1)} ,$$
and $\ind \ms_{(a_1;\ell_1)} \boxtimes \ms_{(a_2;\ell_2)}$ is irreducible.
\item If $a_2 \ge a_1$ and $a_1+\ell_1 \ge a_2 + \ell_2$, then
$$ \ind \ms_{(a_1;\ell_1)} \boxtimes \ms_{(a_2;\ell_2)} \cong
\ind \ms_{(a_2;\ell_2)} \boxtimes \ms_{(a_1;\ell_1)},$$
and $\ind \ms_{(a_1;\ell_1)} \boxtimes \ms_{(a_2;\ell_2)}$ is irreducible.
%\item If $a_2 > a_1$, $\ell_1 + a_1 - 1 > a_2$ and $\ell_2 + a_2 \ge a_1$, then there is a simple module $N$ such that
%$$ 0 \to \ind \Delta_{(a_1; \ell_2+a_2-a_1)} \boxtimes \Delta_{(a_2;\ell_1+a_1-a_2)} \to \ind \Delta_{(a_1;\ell_1)}\boxtimes \Delta_{(a_2;\ell_2)} \to N \to 0 $$
%is exact.

\end{enumerate}
\end{Lem}
\begin{proof} Let $\alpha := \alpha_{(a_1;\ell_1)} + \alpha_{(a_2;\ell_2)}$ and let
$\Sigma_{\ell_1 + \ell_2 }/ \Sigma_{\ell_1} \times \Sigma_{\ell_2}$
be the set of the minimal length coset representatives of
$\Sigma_{\ell_1} \times \Sigma_{\ell_2}$ in $\Sigma_{\ell_1 + \ell_2
}$.

(1) The condition $a_1+\ell_1 - 1 < a_2$ can be visualized as
follows. \vskip 1em
\begin{center}
\begin{texdraw}
\drawdim em \setunitscale 0.15  \linewd 0.6
\arrowheadtype t:F

\move(10 0) \lvec(10 -65)  % \move(0 -60) \lvec(200 -60)
\linewd 0.3
\move(8 -58) \lvec(12 -58)  \htext(-3 -59){\scriptsize$a_1$}
   \htext(0 -51){\scriptsize$\vdots$}
\move(8 -38) \lvec(12 -38)  \htext(-20 -40){\scriptsize$a_1+\ell_1-1$}
\move(8 -28) \lvec(12 -28)  \htext(-3 -29){\scriptsize$a_2$}
   \htext(0 -21){\scriptsize$\vdots$}
\move(8 -8) \lvec(12 -8)  \htext(-20 -10){\scriptsize$a_2+\ell_2-1$}

\move(30 -38)
\lvec(30 -45)
\move(30 -52)
\lvec(30 -58)
\htext(28.1 -40){$\bullet$}
\htext(29 -52){$\vdots$}
\htext(28.1 -60){$\bullet$}

\move(50 -8)
\lvec(50 -15)
\move(50 -22)
\lvec(50 -28)
\htext(48.1 -10){$\bullet$}
\htext(48.9 -22){$\vdots$}
\htext(48.1 -30){$\bullet$}

\lpatt(0.5 1)
\move(30 -38)\clvec(30 -38)(35 -40)(35 -48)
\move(30 -58)\clvec(30 -58)(35 -55)(35 -48)
\htext(37 -50){$\ell_1$}

\move(50 -8)\clvec(50 -8)(55 -10)(55 -18)
\move(50 -28)\clvec(50 -28)(55 -25)(55 -18)
\htext(57 -20){$\ell_2$}

\end{texdraw}
\end{center}
\vskip 0.5em

By \cite[Proposition 2.18]{KL09}, we have
$$ \ch( \ind \ms_{(a_1;\ell_1)} \boxtimes \ms_{(a_2;\ell_2)}) =
\sum_{w \in \Sigma_{\ell_1 + \ell_2 }/ \Sigma_{\ell_1} \times
\Sigma_{\ell_2}} w\cdot (\mathbf{i}_{(a_1;\ell_1)} *
\mathbf{i}_{(a_2;\ell_2)}). $$ Note that each term in $\ch( \ind
\ms_{(a_1;\ell_1)} \boxtimes \ms_{(a_2;\ell_2)})$ has multiplicity
1. Let $Q$ be a quotient of $\ind \ms_{(a_1;\ell_1)} \boxtimes
\ms_{(a_2;\ell_2)}$. It follows from Lemma \ref{Lem: char} that
$\ch(Q)$ contains $\mathbf{i}_{(a_1;\ell_1)} *
\mathbf{i}_{(a_2;\ell_2)}$ as a term. By $\eqref{Eq: Serre
|i-j|>1}$, all terms in $\ch(\ind \ms_{(a_1;\ell_1)} \boxtimes
\ms_{(a_2;\ell_2)})$ occur in $\ch(Q)$. Therefore, $\ind
\ms_{(a_1;\ell_1)} \boxtimes \ms_{(a_2;\ell_2)}$ is irreducible. In
the same manner, one can prove that $ \ind \ms_{(a_2;\ell_2)}
\boxtimes  \ms_{(a_1;\ell_1)}$ is irreducible. Comparing the
characters $ \ch(\ind \ms_{(a_1;\ell_1)} \boxtimes
\ms_{(a_2;\ell_2)} ) $ and $\ch( \ind \ms_{(a_2;\ell_2)} \boxtimes
\ms_{(a_1;\ell_1)})$, by \cite[Theorem 3.17]{KL09}, we conclude
$$ \ind \ms_{(a_1;\ell_1)} \boxtimes \ms_{(a_2;\ell_2)} \cong \ind \ms_{(a_2;\ell_2)}
\boxtimes \ms_{(a_1;\ell_1)} .$$

(2) The conditions $a_2 \ge a_1$ and $a_1+\ell_1 \ge a_2 + \ell_2$
can be visualized as follows. \vskip 1em
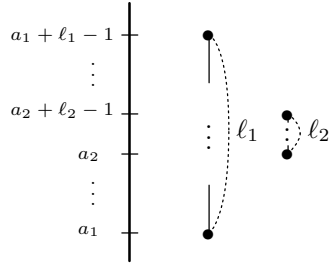
\begin{center}
\begin{texdraw}
\drawdim em \setunitscale 0.15  \linewd 0.6
\arrowheadtype t:F

\move(10 0) \lvec(10 -65)  % \move(0 -60) \lvec(200 -60)
\linewd 0.3
\move(8 -58) \lvec(12 -58)  \htext(-3 -59){\scriptsize$a_1$}
   \htext(0 -51){\scriptsize$\vdots$}
\move(8 -38) \lvec(12 -38)  \htext(-3 -40){\scriptsize$a_2$}
\move(8 -28) \lvec(12 -28)  \htext(-20 -29){\scriptsize$a_2+\ell_2-1$}
   \htext(0 -21){\scriptsize$\vdots$}
\move(8 -8) \lvec(12 -8)  \htext(-20 -10){\scriptsize$a_1+\ell_1-1$}

\move(30 -8)
\lvec(30 -20)
\move(30 -46)
\lvec(30 -58)
\htext(28.1 -10){$\bullet$}
\htext(29 -37){$\vdots$}
\htext(28.1 -60){$\bullet$}

\move(50 -28)
\lvec(50 -30)
\move(50 -36)
\lvec(50 -38)
\htext(48.1 -30){$\bullet$}
\htext(48.9 -35){$\vdots$}
\htext(48.1 -40){$\bullet$}

\lpatt(0.5 1)
\move(30 -8)\clvec(30 -8)(35 -10)(35 -32)
\move(30 -58)\clvec(30 -58)(35 -55)(35 -32)
\htext(37 -35){$\ell_1$}

\move(50 -28)\clvec(50 -28)(53 -30)(53 -33)
\move(50 -38)\clvec(50 -38)(53 -35)(53 -33)
\htext(55 -35){$\ell_2$}

\end{texdraw}
\end{center}
\vskip 0.5em

Let
$$ \mathbf{k} := (a_1, a_1+1, \ldots, a_2,a_2,a_2+1,a_2+1,\ldots, a_2+\ell_2-1, a_2+\ell_2-1,
\ldots, a_1+\ell_1-1  ) \in I^{\alpha}. $$ By Proposition \ref{Prop: shuffle} and the identity
$$ \ch( \ind \ms_{(a_1;\ell_1)} \boxtimes \ms_{(a_2;\ell_2)}) =
\sum_{w \in \Sigma_{\ell_1 + \ell_2 }/ \Sigma_{\ell_1} \times
\Sigma_{\ell_2}} w\cdot (\mathbf{i}_{(a_1;\ell_1)} *
\mathbf{i}_{(a_2;\ell_2)}), $$ it is easy to see that $\mathbf{k}$
occurs in $\ch( \ind \ms_{(a_1;\ell_1)} \boxtimes
\ms_{(a_2;\ell_2)})$ with multiplicity $ 2^{\ell_2} $. On the other
hand, by Lemma \ref{Lem: char}, for any quotient $Q$ of $\ind
\ms_{(a_1;\ell_1)} \boxtimes \ms_{(a_2;\ell_2)}$, $\ch(Q)$ contains
$\mathbf{i}_{(a_1;\ell_1)} * \mathbf{i}_{(a_2;\ell_2)}$ as a term.
By $\eqref{Eq: Serre |i-j|>1}$, $\ch(Q)$ must have the following
term
$$ (a_1, \ldots, a_2, a_2+1, a_2, a_2+2, \ldots, a_1+\ell_1-1, a_2+1, \ldots, a_2+\ell_2-1 ). $$
Hence  by $\eqref{Eq: Serre |i-j|=1}$ and Proposition \ref{Prop: shuffle}, $\ch(Q)$ contains
$$ (a_1, \ldots, a_2, a_2, a_2+1, a_2+2, \ldots, a_1+\ell_1-1, a_2+1, \ldots, a_2+\ell_2-1 ). $$
Continuing this process repeatedly, $\ch(Q)$ must contain the term
$\mathbf{k}$. By $\eqref{Eq: L(i^m):dimension}$, we deduce that
$\mathbf{k}$ occurs in $\ch(Q)$ with multiplicity $2^{\ell_2}$. In
the same manner, for any submodule $L$ of $\ind \ms_{(a_1;\ell_1)}
\boxtimes \ms_{(a_2;\ell_2)}$, $\ch(L)$ contains $\mathbf{k}$ with
multiplicity $2^{\ell_2}$. Therefore, by Lemma \ref{Lem:
irreduciblity}, we conclude that $\ind \ms_{(a_1;\ell_1)} \boxtimes
\ms_{(a_2;\ell_2)}$ is irreducible.

Similarly, one can prove that $\ind \ms_{(a_2;\ell_2)} \boxtimes
\ms_{(a_1;\ell_1)}$ is irreducible. Comparing the characters of $
\ind \ms_{(a_1;\ell_1)} \boxtimes \ms_{(a_2;\ell_2)}$ and that of
$\ind \ms_{(a_2;\ell_2)} \boxtimes \ms_{(a_1;\ell_1)}$, by
\cite[Theorem 3.17]{KL09}, we obtain
$$ \ind \ms_{(a_1;\ell_1)} \boxtimes \ms_{(a_2;\ell_2)} \cong \ind \ms_{(a_2;\ell_2)}
\boxtimes \ms_{(a_1;\ell_1)} .$$
\end{proof}

\begin{Lem} \label{Lem: duality}\

\begin{enumerate}

\item For $a\in \Z_{>0}$ and $\ell_1 \ge \ell_2 \ge \cdots \ge \ell_k > 0$ with $a+\ell_1-1 \le n$,
 we have
$$ \ind \ms_{(a;\ell_1)} \boxtimes \cdots \boxtimes \ms_{(a;\ell_k)} \cong  ( \ind \ms_{(a;\ell_k)} \boxtimes \cdots \boxtimes \ms_{(a;\ell_1)})^*. $$

\item Let $a_1,\ldots ,a_k \in \Z_{>0}$ and $\ell_1 \ge \ell_2 \ge \cdots \ge \ell_k > 0$. If
$$a_i+\ell_i -1 = a_j+\ell_j - 1 \le n \quad (i\ne j),  $$
then we have
$$ \ind \ms_{(a_1;\ell_1)} \boxtimes \cdots \boxtimes \ms_{(a_k;\ell_k)} \cong  ( \ind \ms_{(a_k;\ell_k)} \boxtimes \cdots \boxtimes \ms_{(a_1;\ell_1)})^*. $$

\end{enumerate}

\end{Lem}
\begin{proof} We first prove the assertion (1). Let
$$\ms_i := \ms_{(a;\ell_i)},  \quad \beta_i := \alpha_{(a;\ell_i)}       %, \quad \mathbf{i}_i := \mathbf{i}_{(a;\ell_i)}
\qquad \text{ for $i=1,\ldots, k$, } $$
 and $\beta := \sum_{i=1}^k \beta_i$. Take a nonzero element $v_i \in \ms_i$ for each $i=1,\ldots, k$. From Lemma \ref{Lem: Frobenius reciprocity}
 and Theorem \ref{Thm: ind and coind}, we have an exact sequence
 $$ 0 \longrightarrow N \longrightarrow  \res_{\beta_1, \ldots, \beta_k}\ind \ms_k \boxtimes \cdots \boxtimes \ms_1
 \buildrel \mathfrak{q} \over \longrightarrow  \ms_1 \boxtimes \cdots \boxtimes \ms_k \longrightarrow 0$$
for some submodule $N$ of $\res_{\beta_1, \ldots, \beta_k}\ind \ms_k \boxtimes \cdots \boxtimes \ms_1$.
Take $\xi \in \ind \ms_k \boxtimes \cdots \boxtimes \ms_1$ such that
 $$\mathfrak{q}(\xi) = v_1\otimes v_2 \otimes \cdots \otimes v_k \in \ms_1 \boxtimes \cdots \boxtimes \ms_k. $$
Let $r_1\otimes \cdots \otimes r_k $ be an element of $ R(\beta_1)\otimes \cdots \otimes R(\beta_k)$ such that $\deg(r_1\otimes \cdots \otimes r_k)>0 $.
By $\eqref{Eq: def of ms}$,
the element $r_1\otimes \cdots \otimes r_k$ annihilates $ \ms_1 \boxtimes \cdots \boxtimes \ms_k$, which implies that
\begin{align} \label{Eq: stable under r}
(r_1\otimes \cdots \otimes r_k) \xi \in N.
\end{align}
We now define a $\C$-linear map $f \in (\ind \ms_k \boxtimes \cdots \boxtimes \ms_1)^*$ by
$$ f(\xi) = 1\quad \text{ and } \quad f(\zeta)=0\  \text{ for } \zeta \in N. $$
Note that $f$ does not depend on the choice of $\xi$, and by $\eqref{Eq: stable under r}$
\begin{align} \label{Eq: C psi}
\C f \cong  \ms_1 \boxtimes \cdots \boxtimes \ms_k.
\end{align}

On the other hand, by a direct computation, we may assume that
$$\xi = y \cdot v_k \otimes \cdots \otimes v_1, $$
where $y$ is the longest element in
$\Sigma_{\HT(\beta)}/\Sigma_{\HT(\beta_k)}\times \cdots \times \Sigma_{\HT(\beta_1)} $.
For any element $$w  \in \Sigma_{\HT(\beta)}/ \Sigma_{\HT(\beta_k)}\times \cdots \times \Sigma_{\HT(\beta_1)}, $$
there exists $w'\in \Sigma_{\HT(\beta)} $ such that $w' w = y$. Then, it follows from
$$ (\psi(w')f) (x) = f(w'x) = \left\{
                                       \begin{array}{ll}
                                         1 & \hbox{ if } x = w \cdot v_k \otimes \cdots \otimes v_1, \\
                                         0 & \hbox{otherwise},
                                       \end{array}
                                     \right.
 $$
that $\{ \psi(w')f \ |\ w \in  \Sigma_{\HT(\beta)}/ \Sigma_{\HT(\beta_k)}\times \cdots \times \Sigma_{\HT(\beta_1)} \}$ is a basis for $(\ind \ms_k \boxtimes \cdots \boxtimes \ms_1)^* $.
Hence the $R(\beta)$-module $(\ind \ms_k \boxtimes \cdots \boxtimes \ms_1)^*$ is generated by $f$.

Define the map
$$ F: \ms_1 \boxtimes \cdots \boxtimes \ms_k  \longrightarrow
 \res_{\beta_1,\ldots,\beta_k}(\ind \ms_k \boxtimes \cdots \boxtimes \ms_1)^* $$
by mapping $v_1\otimes \cdots \otimes v_k$ to $ f $. It follows from $\eqref{Eq: C psi}$ that the map $F$ is an
$R(\beta_1) \otimes \cdots \otimes R(\beta_k)$-homomorphism. By Lemma \ref{Lem: Frobenius reciprocity},
we have the $R(\beta)$-homomorphism
$$ \mathcal{F}: \ind \ms_1 \boxtimes \cdots \boxtimes \ms_k  \longrightarrow
 (\ind \ms_k \boxtimes \cdots \boxtimes \ms_1)^* $$ sending
$r \cdot v_1\otimes \cdots \otimes v_k$ to $ r \cdot f$. Since
$$\dim ( \ind \ms_1 \boxtimes \cdots \boxtimes \ms_k)  = \dim (\ind \ms_k \boxtimes \cdots \boxtimes \ms_1)^*$$
and $(\ind \ms_k \boxtimes \cdots \boxtimes \ms_1)^*$ is generated
by $f$, the map $\mathcal{F}$ is an isomorphism, which proves the
assertion (1).

The assertion (2) can be proved in a similar manner. \end{proof}

\begin{Lem} \label{Lem: irr module coming from YD} \
\begin{enumerate}
\item For $a\in \Z_{>0}$ and $\ell_1 \ge \ell_2 \ge \cdots \ge \ell_k > 0$ with $a+\ell_1-1 \le n$,
$$ \ind \ms_{(a;\ell_1)} \boxtimes \ms_{(a;\ell_2)} \boxtimes \cdots \boxtimes \ms_{(a;\ell_k)} $$
is irreducible.
\item Let $a_1,\ldots ,a_k \in \Z_{>0}$ and $\ell_1 \ge \ell_2 \ge \cdots \ge \ell_k > 0$. Suppose that
$$a_i+\ell_i -1 = a_j+\ell_j - 1 \le n \quad (i\ne j).  $$ Set $b := a_1+\ell_1 -1$ and
$$ M := \ind \ms_{(a_1;\ell_1)} \boxtimes \ms_{(a_2;\ell_2)} \boxtimes \cdots \boxtimes \ms_{(a_k;\ell_k)}.$$
Then we have
\begin{enumerate}
\item $ M $ is irreducible,
\item $ \varepsilon_b (M) =k$,
\item $\e_b^k(M ) $ is isomorphic to $ \ind \ms_{(a_1;\ell_1-1)} \boxtimes \ms_{(a_2;\ell_2-1)} \boxtimes \cdots \boxtimes \ms_{(a_k;\ell_k-1)}.$
\end{enumerate}

\end{enumerate}
\end{Lem}
\begin{proof}
We first prove (2). We will use induction on $\ell_1$. If $\ell_1 =
1$, then our assertion follows from $\eqref{Eq: L(i^m)}$
immediately. Assume that $\ell_1 > 1$. Let
$$N := \ind \ms_{(a_1;\ell_1-1)} \boxtimes \ms_{(a_2;\ell_2-1)} \boxtimes \cdots \boxtimes \ms_{(a_k;\ell_k-1)}.$$
By the induction hypothesis, $N$ is irreducible.
By Lemma \ref{Lem: Frobenius reciprocity}, it follows from
$$ \res_{\alpha_{(a_i;\ell_i)} - \alpha_b, \alpha_b} \ms_{(a_i;\ell_i)}  \cong  \ms_{(a_i;\ell_i - 1)} \boxtimes \ms_{(b;1)} $$
that we get an exact sequence
$$ \ind \ms_{(a_i;\ell_i - 1)} \boxtimes \ms_{(b;1)} \longrightarrow \ms_{(a_i;\ell_i)} \longrightarrow 0 .$$
Since $  L(b^k) \cong \ind \underbrace{ \ms_{(b;1)} \boxtimes \cdots \boxtimes \ms_{(b;1)} }_k ,$
by transitivity of induction and Lemma \ref{Lem: linking} (2),
we have
$$ \ind (N \boxtimes L(b^k)) \cong \ind (N \boxtimes \underbrace{ \ms_{(b;1)} \boxtimes \cdots \boxtimes \ms_{(b;1)}}_k )   \longrightarrow M
 \longrightarrow 0.$$
Hence, from \cite[Lemma 3.7]{KL09}, we conclude that
$$ \varepsilon_b (\hd M) = k, \qquad  \  N  \simeq \e_b^k ( \hd M ),$$
and all the other composition factors $L$ of $M$ have
$\varepsilon_b(L) < k$. On the other hand, from Lemma \ref{Lem:
duality} and Lemma \ref{Lem: linking}, we have
$$ 0 \longrightarrow \hd M \simeq (\hd M)^* \longrightarrow M^* \simeq M,  $$
which yields $ \varepsilon_b( \soc M ) \ge k $. Therefore, $M$ is
irreducible.

Similarly, using the operator $\e_i^{\vee}$ in \cite[(2.19)]{LV09},
one can prove the assertion (1).
\end{proof}

Combining Lemma \ref{Lem: irr module coming from YD} with \eqref{Eq:
1-dim rep from YD} and Lemma \ref{Lem: linking}, we obtain the
following proposition.

\begin{Prop} \label{Prop: irr module coming from YD}
Let $\mu = (\mu_1 \ge \mu_2 \ge \ldots \ge \mu_r>0)$ be a Young
diagram, and $k \in \Z_{>0}$. Assume that $ k + \mu_1  - 1 \le n .$
Then
\begin{enumerate}
\item $ \ind \ms_{\mu}[k]$ is irreducible,
\item $ \ind \ms_{\mu}[k] $ is isomorphic to $\ind \tms_{\mu}[k] .$
\end{enumerate}

\end{Prop}

Let $\lambda = (\lambda_1 \ge \lambda_2 \ge  \cdots \ge\lambda_s > 0
)$ be a Young diagram and let $\Psi_{\lambda} : \mathbf{B}(\lambda)
\longrightarrow \YD^s $ be the injective map defined by \eqref{Eq:
injection Psi}. For a semistandard tableau $T$ of shape $\lambda$,
define
\begin{equation} \label{Eq:ms}
\ms_T :=   \ms_{\mu^{(s)}}[s] \boxtimes \ms_{\mu^{(s-1)}}[s-1]
\boxtimes \cdots \boxtimes
 \ms_{\mu^{(1)}}[1],
\end{equation}
where $\Psi_{\lambda}(T) = (\mu^{(1)}, \ldots, \mu^{(s)})$.

Let
$\mu=(\mu_1\ge \ldots \ge \mu_r>0)$ be a Young diagram, and
$$^t \mu = (c_1 \ge \ldots \ge c_t>0).$$
For $k\in \Z_{>0}$, define
$$ \mathbf{i}(\mu; k) := (\underbrace{k,\ldots, k}_{c_1},\underbrace{k+1, \ldots,k+1}_{c_2}, \ldots,
\underbrace{k+\mu_1-1, \ldots, k+\mu_1-1}_{c_t} ) \in
I^{\alpha_{\mu}[k]}. $$  If $ k + \mu_1  - 1 \le n $, then it
follows from Proposition \ref{Prop: shuffle} and Proposition \ref{Prop: irr
module coming from YD} that
\begin{align} \label{Eq: multiplicity YD}
\text{$\mathbf{i}(\mu; k)$ occurs in $\ch (\ind \ms_{\mu}[k])$ with multiplicity ${^t}\mu !:= c_1! c_2! \cdots c_t! $. }
\end{align}
By Proposition \ref{Prop: shuffle} and $\eqref{Eq: multiplicity YD}$, we
deduce
\begin{align} \label{Eq: multiplicity SST}
\text{$\mathbf{i}(\mu^{(s)}; s)* \cdots *\mathbf{i}(\mu^{(1)}; 1)$ occurs in $\ch (\ind \ms_T)$ with multiplicity ${^t}\mu^{(s)}!\cdots {^t}\mu^{(1)}! $. }
\end{align}

Now we will state and prove one of our main results.

\begin{Thm} \label{Thm: hd indT is irr}
Let $T$ be a semistandard tableau of shape $\lambda$. Then $\hd \ind \ms_T$ is irreducible.
\end{Thm}

\begin{proof}
Let $\Psi_{\lambda}(T) = (\mu^{(1)}, \ldots, \mu^{(s)})$ and let $Q$
be a quotient of $\ind T$. It follows from Proposition \ref{Prop:
irr module coming from YD} that $ (\ind \ms_{\mu^{(s)}}[s])
\boxtimes (\ind \ms_{\mu^{(s-1)}}[s-1]) \boxtimes \cdots \boxtimes
(\ind \ms_{\mu^{(1)}}[1])$ is irreducible. Then, by Lemma \ref{Lem:
Frobenius reciprocity}, we have the following exact sequence
$$0 \longrightarrow (\ind \ms_{\mu^{(s)}}[s]) \boxtimes (\ind \ms_{\mu^{(s-1)}}[s-1]) \boxtimes \cdots \boxtimes
( \ind \ms_{\mu^{(1)}}[1])  \longrightarrow \res_{\alpha_{\mu^{(s)}}[s], \ldots, \alpha_{\mu^{(1)}}[1]} Q,$$
which implies that, by $\eqref{Eq: multiplicity YD}$ and $\eqref{Eq: multiplicity SST}$,
$$
 \text{$\mathbf{i}(\mu^{(s)}; s)* \cdots *\mathbf{i}(\mu^{(1)}; 1)$ occurs in $\ch Q$ with multiplicity ${^t}\mu^{(s)}!\cdots {^t}\mu^{(1)}! $.}
$$
Therefore, our assertion follows from Lemma \ref{Lem:
irreduciblity}.
\end{proof}

Thus we obtain a map $ \mathbf{B}(\lambda) \rightarrow \mathfrak{B}(\infty) \otimes \mathbf{T}^{\lambda} \otimes \mathbf{C}$ given by
$$ T  \mapsto \hd \ind \ms_T \otimes t_\lambda \otimes c \quad (T \in \mathbf{B}(\lambda)).$$
We will show that this map is the strict crystal embedding which maps the maximal vector $T_\lambda$ to $\mathbf{1} \otimes t_\lambda \otimes c$.
Here, $\mathbf{1}$ is the trivial $R(0)$-module.
%Thus we obtain a map $\Phi_{\lambda} : \mathbf{B}(\lambda)
%\rightarrow \mathfrak{B}(\lambda)$ defined by $$\Phi_{\lambda}(T)=
%\text{hd} \text{Ind} T \quad (T \in \mathbf{B}(\lambda)).$$ We will
%show that $\Phi_{\lambda}$ is a crystal isomorphism.

For a Young diagram $\mu = (\mu_1\ge \mu_2\ge \cdots\ge \mu_r > 0)$, let
$$ \mu^{+} := (\mu_1 - 1\ge \mu_2 - 1 \ge \cdots\ge \mu_r - 1 \ge 0). $$
For $k = 1 ,\ldots, n$ with $ k - \mu_1 + 1 \ge 1$, we define
\begin{align*}
\ims_\mu[k] &:= \ms_{(k-\mu_1 + 1; \mu_1)} \boxtimes \ms_{(k-\mu_2 + 1; \mu_2)}
\boxtimes \cdots \boxtimes \ms_{(k-\mu_r + 1; \mu_r)}, \\
\tims_\mu[k] &:=  \ms_{(k-\mu_r + 1; \mu_r)} \boxtimes \ms_{(k-\mu_{r-1} + 1; \mu_{r-1})}
\boxtimes \cdots \boxtimes \ms_{(k-\mu_1 + 1; \mu_1)}.
\end{align*}
Pictorially, the modules $ \ims_{\mu}[k]$ and $\tims_{\mu}[k]$ may
be visualized as follows: \vskip 1em
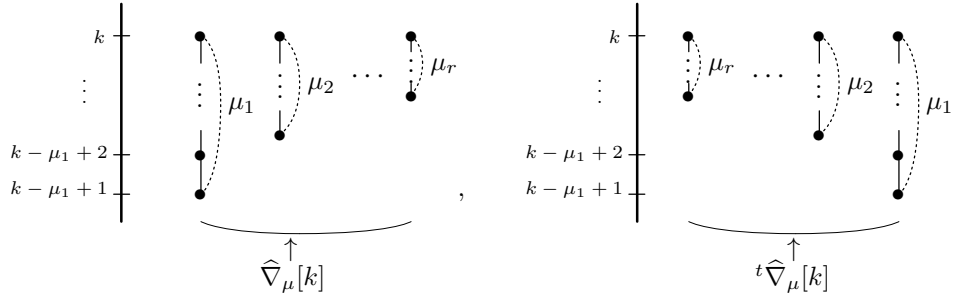
\begin{center}
\begin{texdraw}
\drawdim em \setunitscale 0.15  \linewd 0.6
\arrowheadtype t:F

\move(10 0) \lvec(10 -55)
\linewd 0.3
\move(8 -48) \lvec(12 -48)  \htext(-18 -49){\scriptsize$k-\mu_1+1$}
\move(8 -38) \lvec(12 -38)  \htext(-18 -40){\scriptsize$k-\mu_1+2$}
 \htext(0 -25){\scriptsize$\vdots$}
\move(8 -8) \lvec(12 -8)  \htext(3 -10){\scriptsize$k$}

%mu_1
\move(30 -8)
\lvec(30 -15)
\move(30 -32)
\lvec(30 -48)
\htext(28.3 -10){$\bullet$}
\htext(28.9 -26){$\vdots$}
\htext(28.3 -40){$\bullet$}
\htext(28.3 -50){$\bullet$}

%mu_2

\move(50 -8)
\lvec(50 -15)
\move(50 -27)
\lvec(50 -33)
\htext(48.3 -10){$\bullet$}
\htext(49.1 -24){$\vdots$}
\htext(48.3 -35){$\bullet$}

\htext(68 -20){\large$\cdots$}

%mu_r

\move(83 -8)
\lvec(83 -12)
\move(83 -20)
\lvec(83 -23)
\htext(81.5 -10){$\bullet$}
\htext(82.3 -20){$\vdots$}
\htext(81.5 -25){$\bullet$}

\htext(95 -50){,}

\linewd 0.2
\move(30 -55)\clvec(30 -55)(35 -58)(55 -58)
\move(83 -55)\clvec(83 -55)(78 -58)(55 -58)
\htext(51 -65){$\uparrow$}
\htext(45 -73){$\ims_{\mu}[k]$}
\linewd 0.3

%-----------------------------

\linewd 0.6
\move(140 0) \lvec(140 -55)
\linewd 0.3
\move(138 -48) \lvec(142 -48)  \htext(112 -49){\scriptsize$k-\mu_1+1$}
\move(138 -38) \lvec(142 -38)  \htext(112 -40){\scriptsize$k-\mu_1+2$}
 \htext(130 -25){\scriptsize$\vdots$}
\move(138 -8) \lvec(142 -8)  \htext(133 -10){\scriptsize$k$}

%mu_r

\move(153 -8)
\lvec(153 -12)
\move(153 -20)
\lvec(153 -23)
\htext(151.3 -10){$\bullet$}
\htext(152.1 -20){$\vdots$}
\htext(151.3 -25){$\bullet$}

\htext(169 -20){\large$\cdots$}

%mu_2

\move(186 -8)
\lvec(186 -15)
\move(186 -27)
\lvec(186 -33)
\htext(184.3 -10){$\bullet$}
\htext(185.1 -24){$\vdots$}
\htext(184.3 -35){$\bullet$}

%mu_1
\move(206 -8)
\lvec(206 -15)
\move(206 -32)
\lvec(206 -48)
\htext(204.3 -10){$\bullet$}
\htext(204.9 -26){$\vdots$}
\htext(204.3 -40){$\bullet$}
\htext(204.3 -50){$\bullet$}

\linewd 0.2
\move(153 -55)\clvec(153 -55)(158 -58)(180 -58)
\move(206 -55)\clvec(206 -55)(201 -58)(180 -58)
\htext(178 -65){$\uparrow$}
\htext(170 -73){$\tims_{\mu}[k]$}
\linewd 0.3

%-----------------------------

%mu_1
\lpatt(0.5 1)
\move(30 -8)\clvec(30 -8)(35 -10)(35 -25)
\move(30 -48)\clvec(30 -48)(35 -45)(35 -25)
\htext(37 -28){$\mu_1$}

%mu_2
\move(50 -8)\clvec(50 -8)(55 -10)(55 -20)
\move(50 -33)\clvec(50 -33)(55 -30)(55 -20)
\htext(57 -23){$\mu_2$}

%mu_k
\move(83 -8)\clvec(83 -8)(86 -10)(86 -15)
\move(83 -23)\clvec(83 -23)(86 -20)(86 -15)
\htext(88 -18){$\mu_r$}

%-----------------------------

%mu_k
\move(153 -8)\clvec(153 -8)(156 -10)(156 -15)
\move(153 -23)\clvec(153 -23)(156 -20)(156 -15)
\htext(158 -18){$\mu_r$}

%mu_2
\move(186 -8)\clvec(186 -8)(191 -10)(191 -20)
\move(186 -33)\clvec(186 -33)(191 -30)(191 -20)
\htext(193 -23){$\mu_2$}

%mu_1
\move(206 -8)\clvec(206 -8)(211 -10)(211 -25)
\move(206 -48)\clvec(206 -48)(211 -45)(211 -25)
\htext(213 -28){$\mu_1$}

\end{texdraw}
\end{center}
\vskip 0.5em

By Lemma \ref{Lem: irr module coming from YD} and Lemma \ref{Lem: linking},  we have the following lemma.
\begin{Lem} \label{Lem: inv segment}
 Let $\mu = (\mu_1\ge \cdots\ge \mu_r > 0)$ and $k = 1 ,\ldots, n$ with $ k - \mu_1 + 1 \ge 1$.
\begin{enumerate}
\item $\ind \ims_\mu[k]$ is irreducible.
\item $\ind \ims_\mu[k]$ is isomorphic to $\ind \tims_\mu[k]$.
\item $\varepsilon_k(\ind \ims_\mu[k]) = r$.
\item
$\e_k^r(\ind \ims_\mu[k])  \simeq \ind \ims_{\mu^{+}}[k-1]. $
\end{enumerate}
\end{Lem}

Let $T$ be a semistandard tableau of shape $\lambda=( \lambda_1 \ge
\cdots \ge \lambda_s > 0)$. Suppose that $T$ is not the maximal
vector $T_\lambda$. Write $\Psi_{\lambda}(T) = (\mu^{(1)}, \ldots,
\mu^{(s)})$ and
$$ \mu^{(i)} = ( \mu^{(i)}_1 \ge \mu^{(i)}_2 \ge \cdots \ge \mu^{(i)}_{\lambda_i} \ge 0 ) \quad (i=1,\ldots, s).$$
Recall the notations given in Lemma \ref{Lem: tableau}:
\begin{align*}
  i_T &:= \min \{ \mu_j^{(i)} + i-1 |\  1 \le i \le s,\ 1\le j \le \lambda_i,\ \mu_j^{(i)} > 0 \}, \\
  T^+ &:= \text{ the tableau of shape $\lambda$ obtained from $T$ by replacing the entries } \\
  & \qquad \text{$i_T+1$ by $i_T$ from the top row to the $i_T$th row.}
\end{align*}
We define
\begin{align*}
\overline{\mu}^{(i)} &:= ( \mu_j^{(i)} |\ \mu_j^{(i)} + i-1 \ne i_T) \qquad \text{ for } 1 \le i \le s, \\
\mu_{\rm min}  &:= ( \mu_j^{(i)} |\ \mu_j^{(i)} + i-1 = i_T,\ 1 \le i \le s   ) .
\end{align*}
Note that $\mu_{\rm min}$ is not the empty Young diagram
$(0,0,\ldots)$ and, by construction, for any component $\mu_j^{(i)}
$ in $ \mu_{\rm min} $, we have
\begin{align} \label{Eq: mu-min}
\mu^{(i)}_{j'} \ge \mu^{(i)}_{j}  \quad \text{ for } j' \le j, \qquad  \quad  \mu^{(i')}_{j'}+i' \ge \mu^{(i)}_{j} + i  \quad \text{ for } i' < i.
\end{align}

\begin{Exa}
We use the same notations as in Example \ref{Ex: lambda 1} and
Example \ref{Ex: lambda 2}. Consider the following diagram for
$\Psi_\lambda(T)$:

\vskip 1em
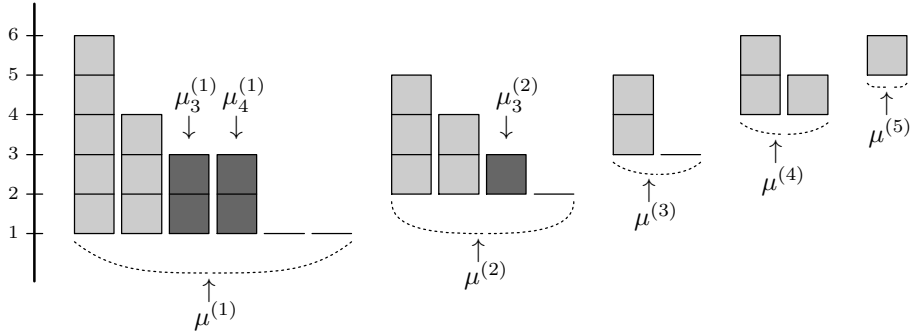
\begin{center}
\begin{texdraw}
\drawdim em \setunitscale 0.15  \linewd 0.6
\arrowheadtype t:F

\move(10 0) \lvec(10 -70)  % \move(0 -60) \lvec(200 -60)
\linewd 0.3
\move(8 -58) \lvec(12 -58)  \htext(3 -59){\scriptsize$1$}
\move(8 -48) \lvec(12 -48)  \htext(3 -49){\scriptsize$2$}
\move(8 -38) \lvec(12 -38)  \htext(3 -39){\scriptsize$3$}
\move(8 -28) \lvec(12 -28)  \htext(3 -29){\scriptsize$4$}
\move(8 -18) \lvec(12 -18)  \htext(3 -19){\scriptsize$5$}
\move(8 -8) \lvec(12 -8)  \htext(3 -9){\scriptsize$6$}

% \mu^(1)

\move(20 -58) \lvec(20 -8) \lvec(30 -8) \lvec(30 -58) \lvec(20 -58) \lfill f:0.8
\move(20 -48) \lvec(30 -48) \move(20 -38) \lvec(30 -38) \move(20 -28) \lvec(30 -28) \move(20 -18) \lvec(30 -18)
\move(32 -58) \lvec(32 -28) \lvec(42 -28) \lvec(42 -58) \lvec(32 -58) \lfill f:0.8
\move(32 -48) \lvec(42 -48) \move(32 -38) \lvec(42 -38)
\move(44 -58) \lvec(44 -38) \lvec(54 -38) \lvec(54 -58) \lvec(44 -58) \lfill f:0.4
\move(44 -48) \lvec(54 -48)
\move(56 -58) \lvec(56 -38) \lvec(66 -38) \lvec(66 -58) \lvec(56 -58) \lfill f:0.4
\move(56 -48) \lvec(66 -48)
\move(68 -58) \lvec(78 -58)
\move(80 -58) \lvec(90 -58)

% \mu^(2)

\move(100 -48) \lvec(100 -18) \lvec(110 -18) \lvec(110 -48) \lvec(100 -48) \lfill f:0.8
\move(100 -38) \lvec(110 -38) \move(100 -28) \lvec(110 -28)
\move(112 -48) \lvec(112 -28) \lvec(122 -28) \lvec(122 -48) \lvec(112 -48) \lfill f:0.8
\move(112 -38) \lvec(122 -38)
\move(124 -48) \lvec(124 -38) \lvec(134 -38) \lvec(134 -48) \lvec(124 -48) \lfill f:0.4
\move(136 -48) \lvec(146 -48)

% \mu^(3)
\move(156 -38) \lvec(156 -18) \lvec(166 -18) \lvec(166 -38) \lvec(156 -38) \lfill f:0.8
\move(156 -28) \lvec(166 -28)
\move(168 -38) \lvec(178 -38)

% \mu^(4)
\move(188 -28) \lvec(188 -8) \lvec(198 -8) \lvec(198 -28) \lvec(188 -28) \lfill f:0.8
\move(188 -18) \lvec(198 -18)
\move(200 -28) \lvec(200 -18) \lvec(210 -18) \lvec(210 -28) \lvec(200 -28) \lfill f:0.8

% \mu^(5)
\move(220 -18) \lvec(220 -8) \lvec(230 -8) \lvec(230 -18) \lvec(220 -18) \lfill f:0.8

% \mu^(1)
\lpatt(0.5 1)
\move(20 -60)\clvec(30 -68)(40 -68)(54 -68)
\move(90 -60)\clvec(80 -68)(70 -68)(55 -68)
\htext(52 -75){$ \uparrow $}
\htext(50 -82){$\mu^{(1)}$}

% \mu^(2)
\move(100 -50)\clvec(100 -58)(115 -58)(123 -58)
\move(146 -50)\clvec(146 -58)(131 -58)(124 -58)
\htext(120 -65){$ \uparrow $}
\htext(118 -72){$\mu^{(2)}$}

% \mu^(3)
\move(156 -40)\clvec(156 -40)(159 -43)(167 -43)
\move(178 -40)\clvec(178 -40)(175 -43)(168 -43)
\htext(163 -50){$ \uparrow $}
\htext(161 -57){$\mu^{(3)}$}

% \mu^(4)
\move(188 -30)\clvec(188 -30)(191 -33)(199 -33)
\move(210 -30)\clvec(210 -30)(207 -33)(200 -33)
\htext(195 -40){$ \uparrow $}
\htext(193 -48){$\mu^{(4)}$}

% \mu^(4)
\move(220 -20)\clvec(221 -21)(221 -21)(225 -21)
\move(230 -20)\clvec(229 -21)(229 -21)(226 -21)
\htext(223 -28){$ \uparrow $}
\htext(220 -36){$\mu^{(5)}$}

\htext(47 -34){$ \downarrow $} \htext(45 -27){$\mu_3^{(1)}$}
\htext(59 -34){$ \downarrow $} \htext(57 -27){$\mu_4^{(1)}$}

\htext(127 -34){$ \downarrow $} \htext(126 -27){$\mu_3^{(2)}$}

\end{texdraw}
\end{center}

\noindent
Then we have
\begin{align*}
\mu_{\rm min} &= (\mu^{(1)}_3, \mu^{(1)}_4, \mu^{(2)}_3), \\
\overline{\mu}^{(1)} &= (\mu^{(1)}_1, \mu^{(1)}_2, \mu^{(1)}_5, \mu^{(1)}_6 ), \quad \overline{\mu}^{(2)} = (\mu^{(2)}_1, \mu^{(2)}_2, \mu^{(2)}_4 ), \\
\overline{\mu}^{(3)} &= (\mu^{(3)}_1, \mu^{(3)}_2 ), \quad \overline{\mu}^{(4)} = (\mu^{(4)}_1, \mu^{(4)}_2 ), \quad \overline{\mu}^{(5)} = (\mu^{(5)}_1).
\end{align*}
Pictorially, the partitions $\mu_{\rm min}$ and
$\overline{\mu}^{(i)}\ (i=1, \ldots,5)$ are given as follows: \vskip
1em
\begin{center}
\begin{texdraw}
\drawdim em \setunitscale 0.15  \linewd 0.6
\arrowheadtype t:F

\move(0 0) \lvec(0 -70)  % \move(0 -60) \lvec(200 -60)
\linewd 0.3

%20 66

\move(-2 -58) \lvec(2 -58)  \htext(-6 -59){\scriptsize$1$}
\move(-2 -48) \lvec(2 -48)  \htext(-6 -49){\scriptsize$2$}
\move(-2 -38) \lvec(2 -38)  \htext(-6 -39){\scriptsize$3$}
\move(-2 -28) \lvec(2 -28)  \htext(-6 -29){\scriptsize$4$}
\move(-2 -18) \lvec(2 -18)  \htext(-6 -19){\scriptsize$5$}
\move(-2 -8) \lvec(2 -8)  \htext(-6 -9){\scriptsize$6$}

\move(6 -58) \lvec(6 -38) \lvec(16 -38) \lvec(16 -58) \lvec(6 -58) \lfill f:0.4
\move(6 -48) \lvec(16 -48)
\move(18 -58) \lvec(18 -38) \lvec(28 -38) \lvec(28 -58) \lvec(18 -58) \lfill f:0.4
\move(18 -48) \lvec(28 -48)
\move(30 -48) \lvec(30 -38) \lvec(40 -38) \lvec(40 -48) \lvec(30 -48) \lfill f:0.4

% \mu^(1)

\move(54 -58) \lvec(54 -8) \lvec(64 -8) \lvec(64 -58) \lvec(54 -58) \lfill f:0.8
\move(54 -48) \lvec(64 -48) \move(54 -38) \lvec(64 -38) \move(54 -28) \lvec(64 -28) \move(54 -18) \lvec(64 -18)
\move(66 -58) \lvec(66 -28) \lvec(76 -28) \lvec(76 -58) \lvec(66 -58) \lfill f:0.8
\move(66 -48) \lvec(76 -48) \move(66 -38) \lvec(76 -38)
\move(78 -58) \lvec(88 -58)
\move(90 -58) \lvec(100 -58)

% \mu^(2)

\move(110 -48) \lvec(110 -18) \lvec(120 -18) \lvec(120 -48) \lvec(110 -48) \lfill f:0.8
\move(110 -38) \lvec(120 -38) \move(110 -28) \lvec(120 -28)
\move(122 -48) \lvec(122 -28) \lvec(132 -28) \lvec(132 -48) \lvec(122 -48) \lfill f:0.8
\move(122 -38) \lvec(132 -38)
\move(134 -48) \lvec(144 -48)

% \mu^(3)
\move(156 -38) \lvec(156 -18) \lvec(166 -18) \lvec(166 -38) \lvec(156 -38) \lfill f:0.8
\move(156 -28) \lvec(166 -28)
\move(168 -38) \lvec(178 -38)

% \mu^(4)
\move(188 -28) \lvec(188 -8) \lvec(198 -8) \lvec(198 -28) \lvec(188 -28) \lfill f:0.8
\move(188 -18) \lvec(198 -18)
\move(200 -28) \lvec(200 -18) \lvec(210 -18) \lvec(210 -28) \lvec(200 -28) \lfill f:0.8

% \mu^(5)
\move(220 -18) \lvec(220 -8) \lvec(230 -8) \lvec(230 -18) \lvec(220 -18) \lfill f:0.8

% \mu^(min)
\lpatt(0.5 1)
\move(6 -60)\clvec(6 -68)(20 -68)(23 -68)
\move(40 -60)\clvec(40 -68)(26 -68)(23 -68)
\htext(23 -75){$ \uparrow $}
\htext(21 -82){$\mu^{(\rm min)}$}

% \mu^(1)
\lpatt(0.5 1)
\move(54 -60)\clvec(54 -68)(70 -68)(78 -68)
\move(100 -60)\clvec(100 -68)(86 -68)(78 -68)
\htext(77 -75){$ \uparrow $}
\htext(75 -82){$\overline{\mu}^{(1)}$}

% \mu^(2)
\move(110 -50)\clvec(110 -58)(125 -58)(128 -58)
\move(146 -50)\clvec(146 -58)(131 -58)(128 -58)
\htext(127 -65){$ \uparrow $}
\htext(125 -72){$\overline{\mu}^{(2)}$}

% \mu^(3)
\move(156 -40)\clvec(156 -40)(159 -43)(167 -43)
\move(178 -40)\clvec(178 -40)(175 -43)(168 -43)
\htext(163 -50){$ \uparrow $}
\htext(161 -57){$\overline{\mu}^{(3)}$}

% \mu^(4)
\move(188 -30)\clvec(188 -30)(191 -33)(199 -33)
\move(210 -30)\clvec(210 -30)(207 -33)(200 -33)
\htext(195 -40){$ \uparrow $}
\htext(193 -48){$\overline{\mu}^{(4)}$}

% \mu^(4)
\move(220 -20)\clvec(221 -21)(221 -21)(225 -21)
\move(230 -20)\clvec(229 -21)(229 -21)(226 -21)
\htext(223 -28){$ \uparrow $}
\htext(220 -36){$\overline{\mu}^{(5)}$}

\end{texdraw}
\end{center}
\end{Exa}

By Proposition \ref{Prop: irr module coming from YD}, Lemma \ref{Lem: linking} (2) and $\eqref{Eq: mu-min}$, we obtain
\begin{align} \label{Eq: ind T}
\ind \ms_T  &\simeq  \ind ( \ms_{\mu^{(s)}}[s] \boxtimes \cdots \boxtimes  \ms_{\mu^{(1)}}[1] ) \nonumber \\
& \simeq \ind ( \tms_{\mu^{(s)}}[s] \boxtimes \cdots \boxtimes  \tms_{\mu^{(1)}}[1] ) \nonumber \\
& \simeq \ind ( \tms_{\overline{\mu}^{(s)}}[s] \boxtimes \cdots \boxtimes  \tms_{\overline{\mu}^{(1)}}[1]
\boxtimes \tims_{\mu_{\rm min }}[i_T]  ).
\end{align}
In the same manner, we have
\begin{align} \label{Eq: ind T^+}
\ind \ms_{T^+}  \simeq \ind ( \tms_{\overline{\mu}^{(s)}}[s] \boxtimes \cdots \boxtimes  \tms_{\overline{\mu}^{(1)}}[1]
\boxtimes \tims_{\mu_{\rm min }^+}[i_T - 1]  ).
\end{align}

Now, we will prove our main result.

\begin{Thm} \label{Thm: crystal iso} \
\begin{enumerate}
\item[(1)] For $T \in \mathbf{B}(\lambda) $, $\hd \ind \ms_T$ is an irreducible $R^\lambda$-module.
\item[(2)] The map $\Phi_\lambda: \mathbf{B}(\lambda) \to
\mathfrak{B}(\lambda)$ defined by
$$ \Phi_\lambda(T) =  \hd \ind \ms_T \qquad (T \in \mathbf{B}(\lambda))$$
is a crystal isomorphism.
\end{enumerate}
\end{Thm}

\begin{proof} Let $\lambda=( \lambda_1 \ge \cdots \ge \lambda_s > 0)$, and let
\begin{align*}
\phi_\lambda : &\  \mathbf{B}(\lambda) \to \mathfrak{B}(\infty) \otimes \mathbf{T}^\lambda \otimes \mathbf{C}, \\
 & \quad  T \ \ \mapsto \ \ \hd \ind \ms_T \otimes t_\lambda \otimes c .
\end{align*}
We first show that $\phi_\lambda$ is the strict crystal embedding which maps the maximal vector $T_\lambda$ to $\mathbf{1} \otimes t_\lambda \otimes c $. Here, $\mathbf{1}$ is the trivial $R(0)$-module.
It is obvious that $\phi_\lambda$ maps $T_\lambda$ to $\mathbf{1} \otimes t_\lambda \otimes c$. If $ T_\lambda = \e_{i_1}^{\rm max} \cdots \e_{i_k}^{ \rm max}  T$
 for $ T \in \mathbf{B}(\lambda) $ and $ i_j \in I$, then it suffices to show that
$$ \e_{i_1}^{\rm max} \cdots \e_{i_k}^{\max} \phi_\lambda(T)  =  \phi_\lambda(\e_{i_1}^{\rm max} \cdots \e_{i_k}^{\rm max}  T). $$
We will use  induction on $\HT( \lambda - \wt(T) )$. If $\wt(T) = \lambda$, then
there is nothing to prove. Assume that $\HT( \lambda - \wt(T) )
> 0 $. Write $\Psi_{\lambda}(T) = (\mu^{(1)}, \ldots, \mu^{(s)})$ and
$$ \mu^{(i)} = ( \mu^{(i)}_1 \ge \mu^{(i)}_2 \ge \cdots \ge \mu^{(i)}_{\lambda_i} \ge 0 ) \quad (i=1,\ldots, s).$$
From $\eqref{Eq: ind T}$, we obtain
\begin{align*}
\ind \ms_T  \simeq \ind ( \tms_{\overline{\mu}^{(s)}}[s] \boxtimes \cdots \boxtimes  \tms_{\overline{\mu}^{(1)}}[1]
\boxtimes \tims_{\mu_{\rm min }}[i_T]  ).
\end{align*}
Let
$$\varepsilon := \varepsilon_{i_T} ( \ind \tims_{\mu_{\rm min }}[i_T] ) \in \Z_{>0}. $$
Since $ \Delta_{i_T} ( \ms_{( i; \mu^{(i)}_j )}  ) =0 $ for any $ \mu^{(i)}_j \in \overline{\mu}^{(i)} $, it follows from Proposition \ref{Prop: shuffle} that
$$ \varepsilon_{i_T} ( \ind ( \tms_{\overline{\mu}^{(s)}}[s] \boxtimes \cdots \boxtimes  \tms_{\overline{\mu}^{(1)}}[1])) = 0. $$
By Lemma \ref{Lem: Frobenius reciprocity}, we have the following nontrivial map:
$$  (\ind  \tms_{\overline{\mu}^{(s)}}[s]) \boxtimes \cdots \boxtimes  (\ind \tms_{\overline{\mu}^{(1)}}[1])
\boxtimes (\ind  \tims_{\mu_{\rm min }}[i_T])  \longrightarrow \res (\hd \ind \ms_T) , $$
which implies that, by Proposition \ref{Prop: irr module coming from YD} and Lemma \ref{Lem: inv segment},
$$ \varepsilon = \varepsilon_{i_T} ( \hd \ind \ms_T ) . $$
Hence, by Lemma \ref{Lem: e_i and tilde e_i}, we have
$$ [e_{i_T}^\varepsilon (\hd \ind \ms_T)] =  q^{-\frac{\varepsilon(\varepsilon-1)}{2}}[\varepsilon]_q!  [\e_{i_T}^\varepsilon (\hd \ind \ms_T)] .$$

On the other hand, by Lemma \ref{Lem: exact}, we obtain
\begin{align*}
e_{i_T} ( \ind \ms_T) & \simeq e_{i_T}( \ind ( \tms_{\overline{\mu}^{(s)}}[s] \boxtimes \cdots \boxtimes  \tms_{\overline{\mu}^{(1)}}[1]
\boxtimes \tims_{\mu_{\rm min }}[i_T]  )) \\
 & \simeq  \ind ( \ind( \tms_{\overline{\mu}^{(s)}}[s] \boxtimes \cdots \boxtimes \tms_{\overline{\mu}^{(1)}}[1])
\boxtimes e_{i_T} \ind (\tims_{\mu_{\rm min }}[i_T])  ).
\end{align*}
It follows from Lemma \ref{Lem: e_i and tilde e_i}, Lemma \ref{Lem: inv segment} and $\eqref{Eq: ind T^+}$ that
\begin{align*}
[e_{i_T}^\varepsilon ( \ind \ms_T)] & \simeq  [\ind ( \ind( \tms_{\overline{\mu}^{(s)}}[s] \boxtimes \cdots \boxtimes  \tms_{\overline{\mu}^{(1)}}[1])
\boxtimes e_{i_T}^{\varepsilon} \ind (\tims_{\mu_{\rm min }}[i_T])  )] \\
&\simeq  q^{-\frac{\varepsilon(\varepsilon-1)}{2}}[\varepsilon]_q![\ind ( \ind( \tms_{\overline{\mu}^{(s)}}[s] \boxtimes \cdots \boxtimes  \tms_{\overline{\mu}^{(1)}}[1])
\boxtimes \e_{i_T}^{\varepsilon} \ind (\tims_{\mu_{\rm min }}[i_T])  )] \\
&\simeq  q^{-\frac{\varepsilon(\varepsilon-1)}{2}}[\varepsilon]_q![\ind ( \ind( \tms_{\overline{\mu}^{(s)}}[s] \boxtimes \cdots \boxtimes  \tms_{\overline{\mu}^{(1)}}[1])
\boxtimes  \ind (\tims_{\mu_{\rm min }^+}[i_T - 1])  )] \\
&\simeq  q^{-\frac{\varepsilon(\varepsilon-1)}{2}}[\varepsilon]_q![\ind ( \tms_{\overline{\mu}^{(s)}}[s] \boxtimes \cdots \boxtimes  \tms_{\overline{\mu}^{(1)}}[1]
\boxtimes  \tims_{\mu_{\rm min }^+}[i_T - 1]  )] \\
&\simeq q^{-\frac{\varepsilon(\varepsilon-1)}{2}}[\varepsilon]_q![\ind \ms_{T^+}].
\end{align*}
Since $e_{i_T}$ is an exact functor, we obtain an exact sequence
$$ e_{i_T}^\varepsilon ( \ind \ms_T) \longrightarrow  e_{i_T}^\varepsilon (\hd \ind \ms_T) \longrightarrow 0, $$
which yields that $ \hd \ind \ms_{T^+} \simeq  \e_{i_T}^\varepsilon (\hd \ind \ms_T) $.
By Lemma \ref{Lem: tableau}, we conclude
\begin{align*}
\phi_\lambda( \e_{i_T}^\varepsilon T ) &= \hd  \ind ( \ms_{\e_{i_T}^\varepsilon T} ) \otimes t_\lambda \otimes c \\
& \simeq \hd  \ind  \ms_{T^+}  \otimes t_\lambda \otimes c \\
& \simeq
\e_{i_T}^\varepsilon ( \hd  \ind  \ms_T ) \otimes t_\lambda \otimes c \\
& =  \e_{i_T}^\varepsilon \phi_\lambda( T ).
\end{align*}
By induction hypothesis, $\phi_\lambda$ is the strict crystal embedding. Therefore, our assertions (1) and (2) follow
from the crystal embedding  $ \mathfrak{B}(\lambda) \to \mathfrak{B}(\infty) \otimes \mathbf{T}^\lambda\ (M \mapsto  \infl_\lambda M \otimes t_\lambda  )$ given in \cite[(5.10)]{LV09}.

\end{proof}

As a result, the set $\{ \hd \ind \ms_T \mid T \in
\mathbf{B}(\lambda) \}$ gives a complete list of irreducible graded
$R$-modules up to isomorphism and grading shift.

Let us  construct the inverse morphism $ \Theta_\lambda:
\mathfrak{B}(\lambda) \to \mathbf{B}(\lambda) $ of $\Phi_{\lambda}$.
The following lemma is crucial.

\begin{Lem} \label{Lem: M is hd ind T}
Let $M$ be an irreducible graded $R^\lambda(\alpha)$-module and let
$T$ be a semistandard tableau in $\mathbf{B}(\lambda)$. Then the
following are equivalent.
\begin{enumerate}
\item $M$ is isomorphic to $\hd \ind \ms_{T}$.
\item $\wt(T) = \lambda - \alpha$ and $ \dim \HOM(\ms_T, \res M) \ne 0$.
\end{enumerate}
\end{Lem}

\begin{proof}
Assume that $M$ is isomorphic to $\hd \ind \ms_{T}$. Clearly,
$\wt(T) = \lambda - \alpha$. Moreover, by Lemma \ref{Lem: Frobenius
reciprocity}, we have
$$ \dim \HOM(\ms_T, \res M) \ne 0.$$

Conversely, suppose that $\wt(T) = \lambda - \alpha$ and $ \dim
\HOM(\ms_T, \res M) \ne 0$. Then we have a nontrivial map
$$ \ind \ms_T \longrightarrow M \longrightarrow 0. $$
Then it follows from Theorem \ref{Thm: hd indT is irr} that $ M
\cong  \hd \ind \ms_T . $
\end{proof}

Given an irreducible $R^{\lambda}(\alpha)$-module $M$, we take $T_M
\in \mathbf{B}(\lambda)$ such that
$$\wt(T_M) = \lambda - \alpha \quad \text{ and }\quad   \dim \HOM(\ms_{T_M}, \res M) \ne 0 . $$
By Theorem \ref{Thm: crystal iso} and Lemma \ref{Lem: M is hd ind
T}, the tableau $T_M$ is well-defined. Now, it is straightforward to
verify that $\Phi_{\lambda}$ and $\Theta_{\lambda}$ are inverses to
each other.

\begin{Prop}
The map defined by $ \Theta_\lambda : \mathfrak{B}(\lambda) \to
\mathbf{B}(\lambda)$ defined by
$$ \Theta_\lambda (M) = T_M \quad (M \in \mathfrak{B}(\lambda)) $$
is the inverse morphism of $\Phi_{\lambda}$.
\end{Prop}

\vskip 3em

\section{Irreducible $R$-modules and marginally large tableaux} \label{Sec: KL alg and B(infty)}

In this section, using the results proved in Section \ref{Sec: KL
alg and B(lambda)}, we construct an explicit crystal isomorphism
$\Phi_{\infty} : \mathbf{B}(\infty) \overset{\sim} \longrightarrow
\mathfrak{B}(\infty)$. Consequently, we obtain a complete list of
irreducible graded $R$-modules up to isomorphism and grading shift.

Let us recall the map $\Psi_{\infty}: \mathbf{B}(\infty) \to
\mathcal{Y}^n$, $T \mapsto (\mu^{(1)}, \ldots, \mu^{(n)})$ defined
in \eqref{Eq: injection Psi infity}. For $T \in \mathbf{B}(\infty)$,
we define
$$ \ms_T := \ms_{\mu^{(n)}}[n] \boxtimes \ms_{\mu^{(n-1)}}[n-1] \boxtimes \cdots \boxtimes \ms_{\mu^{(1)}}[1]. $$
By Lemma \ref{Lem: Psi B(infty)}, if $\iota_\lambda(T') = T \otimes
t_\lambda \otimes c$ for $T' \in \mathbf{B}(\lambda)$, then we have
\begin{align} \label{Eq: ms of T and Tml}
\ms_{T'} = \ms_T .
\end{align}
%Then we have the following theorem.

\begin{Thm}  \label{Thm: crystal iso B(infty)}
The map $\Phi_\infty: \mathbf{B}(\infty) \to \mathfrak{B}(\infty) $
defined by
$$ \Phi_\infty(T) := \hd \ind \ms_T \quad (T \in \mathbf{B}(\infty))$$
is a crystal isomorphism.
\end{Thm}
\begin{proof}
It is obvious that $\Phi_\infty$ maps the maximal vector $T_\infty$
of $\mathbf{B}(\infty)$ to the maximal vector $\mathbf{1}$ of $
\mathfrak{B}(\infty)$. Here, $\mathbf{1}$ is the trivial
$R(0)$-module. Take a tableau $T \in \mathbf{B}(\infty)$ and suppose
that $ T_\infty = \e_{i_1}^{j_1} \e_{i_2}^{j_2} \cdots
\e_{i_t}^{j_t}T $ for some $i_k = 1, \ldots , n,\  j_k \in \Z_{>0}$.
Then it suffices to show that
$$ \e_{i_1}^{j_1} \e_{i_2}^{j_2} \cdots \e_{i_t}^{j_t} \Phi_\infty(T)  =  \Phi_\infty(\e_{i_1}^{j_1} \e_{i_2}^{j_2} \cdots \e_{i_t}^{j_t}  T ) .   $$
Take a dominant integral weight $\lambda \in P^{+}$ with
$\lambda(h_i) \gg 0$ for all $i \in I$ so that one can find $T' \in
\mathbf{B}(\lambda)$ satisfying
$$\iota_\lambda(T') = T \otimes t_\lambda \otimes c,$$
where $\iota_\lambda: \mathbf{B}(\lambda) \to \mathbf{B}(\infty)
\otimes \mathbf{T}^\lambda \otimes \mathbf{C}$ is the crystal
embedding given in \eqref{Eq: Crystal embedding}. Note that $
T_\lambda =  \e_{i_1}^{j_1} \e_{i_2}^{j_2} \cdots \e_{i_t}^{j_t}T'
$. Hence it follows from Theorem \ref{Thm: crystal iso} and
$\eqref{Eq: ms of T and Tml}$ that
\begin{align*}
\Phi_\infty(\e_{i_1}^{j_1} \e_{i_2}^{j_2} \cdots \e_{i_t}^{j_t}T) &=
\Phi_\infty(T_\infty) = \Phi_\lambda(T_\lambda)
= \e_{i_1}^{j_1} \e_{i_2}^{j_2} \cdots \e_{i_t}^{j_t} \Phi_\lambda(T') \\
&= \e_{i_1}^{j_1} \e_{i_2}^{j_2} \cdots \e_{i_t}^{j_t} (\hd \ind \ms_{T'})
= \e_{i_1}^{j_1} \e_{i_2}^{j_2} \cdots \e_{i_t}^{j_t} (\hd \ind \ms_{T}) \\
&=  \e_{i_1}^{j_1} \e_{i_2}^{j_2} \cdots \e_{i_t}^{j_t}
\Phi_\infty(T),
\end{align*}
which completes the proof.
\end{proof}

We now construct the inverse map
$$ \Theta_\infty: \mathfrak{B}(\infty) \to \mathbf{B}(\infty)  $$
of the crystal isomorphism $\Phi_\infty$. Given an irreducible $R(\alpha)$-module $M$, we take $T_M \in \mathbf{B}(\infty)$ such that
$$ \wt(T_M) = - \alpha \quad \text{ and } \quad \dim \HOM(\ms_{T_M}, \res M) \ne 0. $$
By Lemma \ref{Lem: M is hd ind T} and Theorem \ref{Thm: crystal iso
B(infty)}, we obtain the following proposition.
\begin{Prop}
The map $\Theta_\infty: \mathfrak{B}(\infty) \to \mathbf{B}(\infty)$
defined by
$$ \Theta_\infty(T) = T_M \quad (T \in \mathfrak{B}(\infty))  $$
is the inverse morphism of $\Phi_\infty$.
\end{Prop}

\vskip 5em

%%%%%%%%%%%%%%%%%%%%%%%%%%%%%%%%%%%%%%%%%%%%%%%%%%%%%%%%%%%%%%%%%%%

\bibliographystyle{amsplain}
\bibliography{tableaux}

%%%%%%%%%%%%%%%%%%%%%%%%%%%%%%%%%%%%%%%%%%%%%%%%%%%%%%%%%%%%%%%%%%%

\end{document}